\documentclass{siamart190516}
\usepackage{hyperref}
\usepackage{amsmath, amssymb}

\usepackage{xcolor}

\usepackage{verbatim}

\usepackage{lipsum}
\usepackage{amsfonts}
\usepackage{graphicx}
\usepackage{epstopdf}
\usepackage{algorithmic}
\ifpdf
  \DeclareGraphicsExtensions{.eps,.pdf,.png,.jpg}
\else
  \DeclareGraphicsExtensions{.eps}
\fi

\newsiamremark{remark}{Remark}
\newsiamremark{hypothesis}{Hypothesis}
\crefname{hypothesis}{Hypothesis}{Hypotheses}
\newsiamthm{claim}{Claim}

\usepackage{amsopn}

\def\dd{{\rm d}}

\def\weight(#1,#2){c_{#1,#2}}

\def\ra{{\rm ra}}

 \if {
  
} \fi

\if {

}{{\caption}}
} \fi

\def\vh{\hat{v}}

\def\pb{\bar{p}}

\def\vb{\bar{v}}

\def\xb{\bar{x}}
\def\yb{\bar{y}}

\def\mt{\tilde{m}}

\def\tt{\tilde{t}}

\def\calk{{\mathcal K}}

\def\calp{{\mathcal P}}

\def\cR{{\mathcal R}}

\def\eps{\varepsilon}

\def\1B{{\bf  1}}

\def\argmin{\mathop{\rm argmin}}

\def\supp{\mathop{\rm supp}}

\def\Min{\mathop{\rm Min}}

\def\half{\mbox{$\frac{1}{2}$}}

\def\1B{{\bf  1}}

\newcommand{\NN}{\mathbb{N}}

\newcommand{\RR}{\mathbb{R}}

\def\NN{\mathbb{N}}

\def\cR{\mathbb{R}}

\newcommand\be{\begin{equation}}
\newcommand\ee{\end{equation}}
\newcommand\ba{\begin{array}}
\newcommand\ea{\end{array}}
\newcommand{\bea}{\begin{eqnarray}}
\newcommand{\eea}{\end{eqnarray}}
\newcommand{\bean}{\begin{eqnarray*}}
\newcommand{\eean}{\end{eqnarray*}}
\newcommand\bes{\begin{equation*}}
\newcommand\ees{\end{equation*}}

\def\rar{\rightarrow}

\def\la{\langle}
\def\ra{\rangle}

\newcommand{\hypConv}{(H1)}
\newcommand{\hypConvL}{(i)}
\newcommand{\hypConvC}{(ii)}
\newcommand{\hypConvPsi}{(iii)}

\newcommand{\hypReg}{(H2)}
\newcommand{\hypRegL}{(i)}
\newcommand{\hypRegA}{(ii)}
\newcommand{\hypRegF}{(iii)}
\newcommand{\hypRegPsi}{(iv)}

\newcommand{\hypBound}{(H3)}
\newcommand{\hypBoundL}{(i)}
\newcommand{\hypBoundC}{(ii)}
\newcommand{\hypBoundA}{(iii)}
\newcommand{\hypBoundSupport}{(iv)}
\newcommand{\hypBoundF}{(v)}
\newcommand{\hypBoundPsi}{(vi)}

\newcommand{\hypFeas}{(H4)}
\newcommand{\hypFeasLocal}{(i)}
\newcommand{\hypFeasGlobal}{(ii)}

\newcommand{\hypQualif}{(H5)}
\newcommand{\hypQualifEq}{(i)}
\newcommand{\hypQualifIneq}{(ii)}
\newcommand{\hypQualifLI}{(iii)}
\newcommand{\hypQualifInward}{(iv)}

\newcommand{\controlset}{\mathcal{U}}

\headers{Aggregative mean field games of controls}{J.F.\@  Bonnans, J.\@ Gianatti, and L.\@ Pfeiffer}

\title{A Lagrangian approach for aggregative mean field games of controls with mixed and final constraints \thanks{
The first author was partially supported by the FiME Lab Research Initiative (Institut Europlace de Finance).
This article benefited from the support of the FMJH Program PGMO and from the support to this program from EDF.
}
}

\author{J. Fr\'{e}d\'{e}ric Bonnans\thanks{Université Paris-Saclay, CNRS, CentraleSupélec, Inria, Laboratoire des signaux et systèmes, 91190, Gif-sur-Yvette, France
(\email{frederic.bonnans@inria.fr}, \email{laurent.pfeiffer@inria.fr}).}
\and Justina Gianatti\thanks{CIFASIS-CONICET-UNR, Ocampo y Esmeralda, S2000EZP, Rosario, Argentina\\ (\email{gianatti@cifasis-conicet.gov.ar}).}
\and Laurent Pfeiffer\footnotemark[2]}

\begin{document}

\maketitle

\begin{abstract}
The objective of this paper is to analyze the existence of equilibria for a class of deterministic mean field games of controls. The interaction between players is due to
both a congestion term and a price function which depends on the distributions of the optimal strategies. Moreover, final state and mixed state-control constraints are considered, the dynamics being nonlinear and affine with respect to the control. The existence of equilibria is obtained by Kakutani's theorem, applied to a fixed point formulation of the problem. Finally, uniqueness results are shown under monotonicity assumptions.
\end{abstract}




\maketitle
\section{Introduction}

In this article we consider a Nash equilibrium problem involving a large number $N$ of agents, each of them solving a deterministic optimal control problem involving control-affine nonlinear dynamics, final state constraints, and mixed state-control constraints. The agents may only differ from each other by their initial condition. The interaction between the agents is induced by a price variable and a congestion term, which are determined by the collective behavior of the agents.
Our mathematical analysis focuses on an equilibrium problem which models the asymptotic limit when $N$ goes to infinity and when each isolated agent is supposed to have no impact on the coupling terms (the price variable and the congestion term).
Therefore the problem falls into the class of mean field games (MFGs), which have received considerable attention in the literature since their introduction in the pioneering works by Lasry and Lions \cite{LASRY2006619,LASRY2006679,Lasry2007} and Caines, Huang and Malham\'e  \cite{huang2006}.

Our work addresses two main difficulties. The first difficulty of our model is the interaction induced by the price variable. In the cost function of each agent, the price penalizes linearly the control variable. It is defined as a monotonic function of some aggregative term that can be interpreted as a demand. Here it is the average value of the controls exerted by all agents. This kind of interaction is similar to the one in Cournot models in economics, where companies without market power compete on the amount of some product. Our model is representative from games in energy markets involving a large number of small storage devices and some endogenous price depending on the average speed of charge of the devices. See for instance \cite{alasseur20,depaola15,liu20, paccagnan16}.
The second difficulty is the presence of mixed control-state constraints and final state constraints. They appear naturally in applications in electrical engineering: for example, when the storage devices must be fully (or partially) loaded at the end of the time frame. 
In the appendix, we motivate the use of mixed constraints with an example involving gas storages. 

In most MFG models proposed in the literature, the agents interact only through their position (their state variable). Mean field game models with interaction through the states and controls are now commonly called \emph{MFGs of controls}. The terminologies \emph{extended MFGs} and \emph{strongly coupled MFGs} are also employed.
Let us review the articles dedicated to such models.
In \cite{GOMES201449}, a stationary second order MFG of controls is studied. A deterministic MFG of controls is considered in \cite{GomesVoskanyan16}. An existence result has been obtained for a quite general MFG model in \cite{Cardaliaguet:2018aa}. A uniqueness result is provided in \cite{bertucci2018remarks}.
The works \cite{kobeissi2020mean,kobeissi2020classical} analyse the existence and uniqueness of classical solutions in the second order case. An existence result is provided in the monograph \cite[Section 4.6]{carmona18}, for MFGs described by forward backward stochastic differential equations.
The particular price interaction investigated in this article has been studied in \cite{Bonnans2019} in the second order case and in \cite{graber2020weak} in the case of a degenerate diffusion and potential congestion terms.

Most MFG models consist of a coupled system of partial differential equations (PDEs), the Fokker-Planck equation and the Hamilton-Jacobi-Bellman (HJB) equation.
The presence of final and mixed constraints in the underlying optimal control problem makes it difficult to characterize the behavior of a representative agent with the classical HJB approach. We therefore rely on a Lagrangian formulation of the problem, rather than on a PDE approach. More precisely, our equilibrium problem is posed on the set of Borel probability measures on the space of state-control trajectories.
The Lagrangian approach has been employed in several references dealing with deterministic MFGs. Variational MFGs are studied in \cite{Benamou2017}. The article \cite{MazantiSantambrogio} deals with minimal-time MFGs.  The three articles \cite{Cannarsa-Capuani, CCC11, cannarsa2018mean} deal with state-constrained MFGs and with the connection between the Lagrangian and the PDE formulations. In \cite{cannarsa2019mild} MFGs with linear dynamics are considered and in  \cite{achdou2021deterministic} state-constrained MFGs with control on the acceleration are studied.

At a methodological level, the common feature of almost all studies dedicated to MFGs of controls is the introduction of an auxiliary mapping, which allows to put the equilibrium problem in a reduced form that can be handled with a fixed point approach.
In the PDE approach, the auxiliary mapping allows to express the control of a representative agent at a given time $t$ in function of its current state $x$, the equilibrium distribution (of the states) and the gradient of the value function (see for example \cite[Lemma 4.60]{carmona18}, \cite[Lemma 5.2]{Cardaliaguet:2018aa} or \cite[Lemma 5]{Bonnans2019}).
This relation is in general not explicit, contrary to MFGs with interaction through the state variable only.
In the probabilistic approach of \cite[Assumption G]{GomesVoskanyan16}, the auxiliary mapping depends on $t$, $x$, and a pair of random variables $(X_t,P_t)$, whose distribution coincides with the distribution of pairs of state-costate of all agents in the game. In \cite[Lemma 4.61]{carmona18}, the auxiliary mapping directly depends on the distribution of $(X_t,P_t)$.
Our roadmap is the same as the one used in the references mentioned above: we introduce an auxiliary mapping (of the same nature as the one in \cite{carmona18}) which allows to write the equilibrium problem in a reduced form which is then tractable with a fixed point argument. After reformulation, the equilibrium problem is posed on the set of Borel probability measures on the space of state-costate trajectories.

Our article is 
 one of the very few publications dealing with first order MFGs of controls and Lagrangian formulation for these problems. 
(i) The article of Gomes and Voskanyan \cite{GomesVoskanyan16} is the closest to our work. Their analysis relies in a quite crucial manner on some regularity properties of the value function associated with the underlying optimal control problem (Lipschitz continuity, semiconcavity) which are easily demonstrated in their framework without constraints.
Those properties are not needed in the Lagrangian framework. They could probably be established, but under stronger qualification conditions than those in force in the present work.
Incidentally, the initial distribution of the agents must have a density in \cite{GomesVoskanyan16}, which is not the case in the present work.
(ii) Carmona and Delarue have an existence result, for an MFG of controls posed as a forward-backward stochastic differential equation, see \cite[Proposition 4.64]{carmona18}. This model relies on Pontryagin's principle, which is a sufficient condition only under convexity assumptions on the underlying optimal control problem (see the assumption SMP \cite[page 161]{carmona18}), which we do not need. Let us mention that their other result \cite[Proposition 4.64]{carmona18} concerns the second order case.
(iii) In a recent work, Graber, Mullenix and Pfeiffer have obtained the existence of a solution for an MFG of controls formulated as a coupled system of possibly degenerate PDEs. This work is restricted to the potential case, when the local congestion term is the derivative of some convex function. It also relies on a periodicity condition on the data functions, which we do not need here.
(iv) Recently in \cite{Santambrogio-Shim}, the authors study the existence of a Lagrangian equilibrium for an MFG of controls, following a variational approach instead of solving a fixed point problem, as proposed here.


The paper is organized as follows: In Section \ref{sec:2} we present the problem that we address here, referred to as MFGC. We introduce the main notation and we define the notion of Lagrangian equilibria for MFGC that we use throughout this work.
In Section \ref{sec:control-prob} we study the optimal control problem associated with an individual player, providing optimality conditions and regularity of solutions. Defining an auxiliary notion of equlibria, by a fixed point argument, in Section \ref{sec:existence}  we prove the existence of Lagrangian equilibria.
In Section \ref{sec:uniqueness}, under additional monotonicity assumptions we analyze the uniqueness of solutions.


\section{Description of the aggregative MFGC problem}    \label{sec:2}
\subsection{Preliminaries}  \label{sec:preliminaries}

Let $(X,d)$ be a separable metric space. We denote by 
 $\calp(X)$ the set of Borel probability measures on $X$. Given $p\in [1,+\infty)$, it is defined $\calp_p(X)$ as the set of probability measures $\mu$ on $X$ such that
$$
\int_X d(x, x_0)^p\dd\mu(x)< +\infty,
$$
for some (and thus any) $x_0\in X$. The Monge-Kantorovich distance on $\calp_p(X)$ is given by
$$
d_p(\mu, \nu)=\inf_{\pi\in\Pi(\mu, \nu)}\left[\int_Xd(x,y)^p\dd\pi(x,y) \right]^\frac{1}{p},
$$
where $\Pi(\mu, \nu)$ denotes the set of probability measures on $X\times X$ with first and second marginals equal to $\mu$ and $\nu$ respectively. In this paper, we work with $p=1$. For all $\mu, \nu\in \calp_1(X)$, we have the following formula (see \cite[Theorem 11.8.2]{Dudley}):
$$
d_1(\mu, \nu)=\sup\left\{\int_X f(x)\dd \mu(x)-\int_Xf(x)\dd\nu(x)\;|\;f  \colon X\to\cR\;\text{ is }1\text{-Lipschitz}\right\}.$$
%
%

We recall the definition of narrow convergence of measures. We say that the sequence $(\mu_n)_{n\in\NN}\subset \calp(X)$ narrowly converges to $\mu\in\calp(X)$ if
$$
\lim_{n\to\infty}\int_Xf(x)\dd\mu_n(x)=\int_Xf(x)\dd\mu(x), \;\;\;\forall f\in C_b^0(X),
$$
where $C_b^0(X)$ denotes the set of all  continuous and bounded real functions defined on $X$. Throughout this work we endow the space $\calp_1(X)$ with the narrow topology. 
As a consequence of \cite[Proposition 7.1.5]{ambrosio2008gradient}, for any compact set $K\subset X$, we have for all $p\geq 1$, $\calp(K)=\calp_p(K)$  and $d_p$ metricizes the narrow convergence of probability measures on the set $\calp(K)$. In addition, $\calp(K)$ is compact.




\subsection{MFG equilibria and main notation}     \label{sec:MFG-def-notation}


We start by defining  the optimal control problem that each agent aims to solve,  assuming that the price and the distribution of the other players are known. The problem takes the form of a constrained minimization problem parameterized by the initial condition $x_0\in\cR^n$, the agents distribution  $m\in C([0,T];\calp_1(\cR^n))$ and the price $P\in L^\infty(0,T;\cR^m)$. 

Let $\Gamma := H^1(0,T;\cR^n) \subset C(0,T;\cR^n)$ be equipped with 
 the supremum norm, denoted by $\| \cdot \|_\infty$. 
Given $x_0 \in \cR^n$, we define $\Gamma[x_0]$ by
\begin{equation*}
\Gamma[x_0]= \big\{ \gamma \in \Gamma :\, \gamma(0)= x_0 \big\}.
\end{equation*}

We take $L^2(0,T;\cR^m)$ as the control space, which we denote by $\controlset$. We denote by $\mathcal{K}[x_0]$ the feasible set that is defined by
\begin{equation*}
\mathcal{K}[x_0]
= \left\{
\begin{array}{ll}
(\gamma,v) \in \Gamma \times \controlset :  &
\begin{array}{ll}
\dot{\gamma}(t)= a(\gamma(t)) + b(\gamma(t)) v(t), & \text{for a.e.\@ $t \in (0,T),$} \\
\gamma(0)= x_0, & \\
c(\gamma(t),v(t)) \leq 0, & \text{for a.e.\@ $t \in (0,T)$},\\
g_1(\gamma(T))= 0, & \\
g_2(\gamma(T))\leq 0 &
\end{array}
\end{array}
\right\}.
\end{equation*}
The dynamics coefficients are $a \colon \cR^n \rightarrow \cR^n$ and $b \colon \cR^n \rightarrow \cR^{n \times m}$ (note that $b_i(x) \in \cR^n$ will denote the $i$-th column of $b(x)$). The final equality and inequality constraint functions are, respectively, $g_1 \colon \cR^n \rightarrow \cR^{n_{g_1}}$, and $g_2 \colon \cR^n \rightarrow \cR^{n_{g_2}}$, and the state-control constraint function is $c \colon \cR^n \times \cR^m \rightarrow \cR^{n_c}$. Now we define the cost functional $J[m, P] \colon \Gamma\times\controlset\to\cR$ as
\begin{equation*}
J[m,P](\gamma,v)
= \int_0^T \left(L(\gamma(t),v(t)) + \langle P(t), v(t) \rangle + f(\gamma(t),m(t))\right) \dd t + g_0(\gamma(T),m(T)).
\end{equation*}
Here $L \colon \cR^n \times \cR^m \rightarrow \cR$ represents the running cost of the agents, $f \colon\cR^n \times \mathcal{P}_1(\cR^n)\to\cR$, the congestion function, and $g_0 \colon \cR^n \times \mathcal{P}_1(\cR^n) \rightarrow \cR$ is the final cost. Therefore, the optimal control problem that each agent addresses is
\begin{equation} \label{eq:control_prob}
\Min_{(\gamma, v)\in \calk[x_0]} J[m,P](\gamma, v).
\end{equation}
The set of optimal trajectories for this minimization problem is denoted by
\begin{equation} 
\label{def-gam-m-p0-x}
\Gamma[m,P,x_0]
= \left\{ \bar{\gamma} \in \Gamma[x_0]\; :\;
\exists \bar{v} \in \controlset,\, (\bar{\gamma},\bar{v}) \;\;\text{is a solution to \eqref{eq:control_prob}}
\right\}.
\end{equation}

\subsubsection{Lagrangian MFGC equilibria}  \label{subsec:LagrangianMFGCequilibria}

In the previous paragraph, we have described the optimization problem, for a particular player, given the price and the agents distribution. We describe now how the price is related to the collective behavior of all agents and give a Lagrangian description of our mean field game.

Let $m_0\in\calp_1(\cR^n)$ be the initial distribution of the agents.
We fix a price function $\psi \colon \cR^m \rightarrow \cR^m$,  which is assumed to be bounded.
 For $t\in[0,T]$, the mapping $e_t:\Gamma\times   \controlset\to\cR^n$ is given by 
$
e_t(\gamma, v)=\gamma(t)
$. 
We define the set
$$
\calp_{m_0}\left(\Gamma\times  \controlset\right)=\big\{\eta\in\calp_1\left( \Gamma\times \controlset\right)\;:\;e_0\sharp\eta = m_0 \big\}.
$$
Given $\eta\in\calp_{m_0}(\Gamma\times   \controlset)$, we define the cost functional $J^\eta=J[m^\eta, P^\eta]$,  
where the coupling terms $m^\eta \colon t\in[0,T]\mapsto m^\eta_t\in\calp_1(\cR^n)$ and $P^\eta \in L^\infty(0,T;\cR^m)$ are given by
\begin{equation}    \label{eq:def-m-P-eta-Lagrangian}
m_t^\eta= e_t \sharp \eta \quad \text{and} \quad
P^\eta= \Psi \Big( \int_{\Gamma\times \controlset} v \, \dd\eta(\gamma,v) \Big).
\end{equation}
The continuity of the mapping $t\mapsto m^\eta_t$ will be ensured by Lemma \ref{lem:m-mu-hol}. In the definition of $P^\eta$, $\int_{\Gamma\times \controlset} v \,\mathrm{d}\eta(\gamma,v)$ is a Bochner integral with value in $L^2(0,T;\cR^m)$ (which is well defined since $\eta\in\calp_1(\Gamma\times \controlset)$) and the mapping $\Psi \colon \theta \in L^2(0,T;\cR^m) \rightarrow \Psi[\theta] \in L^\infty(0,T;\cR^m)$ denotes the Nemytskii operator associated with the price function $\psi \colon \cR^m\to \cR^m$, defined by $\Psi[\theta](t)= \psi(\theta(t))$, for a.e. $t \in (0,T)$.

Given $x_0 \in \cR^n$ and $\eta \in \calp_{m_0}(\Gamma\times \controlset)$, we denote by $\Gamma^\eta[x_0]$ the set of optimal state-control trajectories associated with the cost $J^\eta$ and set of constraints $\calk[x_0]$:
\begin{equation*}
\Gamma^\eta[x_0]=
\Big\{ (\bar{\gamma},\vb) \in \calk[x_0]  :\, J^\eta(\bar{\gamma}, \vb)\leq J^\eta(\gamma, v)\; \ \forall (\gamma,v) \in \calk[x_0]\Big\}.
\end{equation*}

\begin{definition} \label{def:lagrangian}
We call {\it Lagrangian MFGC equilibrium} any distribution $\eta \in \calp_{m_0}(\Gamma\times \controlset)$ supported on optimal trajectories, i.e.
\begin{equation*}
\supp(\eta) \subset \bigcup_{x \in \supp(m_0)} \Gamma^\eta[x].
\end{equation*}
\end{definition}

The main objective of this paper is to prove the existence of a Lagrangian MFGC equilibrium, under the assumptions described in the following subsection.

\if{
\subsubsection{Auxiliary MFGC equilibria}\label{auxMFGCequ}

In order to analyze the existence of Lagrangian MFGC equilibria, 
we propose here a new notion of equilibrium, that we call auxiliary equilibrium. 
 We set
\begin{equation*}
\tilde{\Gamma}= H^1(0,T;\cR^n) \times H^1(0,T;\cR^n).
\end{equation*}
We equip $\tilde{\Gamma}$ with the supremum norm, defined by 
$ 
\max( \| \gamma \|_\infty, \| p \|_\infty)
$ 
for a given pair $(\gamma,p) \in \tilde{\Gamma}$. We denote it (by extension) $\| (\gamma,p) \|_\infty$.
For any $x_0 \in \cR^n$, we define
\begin{equation*}
\tilde{\Gamma}[x_0]= \big\{ (\gamma,p) \in \tilde{\Gamma} \,:\, \gamma(0)= x_0 \big\}.
\end{equation*}
Given $t \in [0,T]$, we consider the mappings $\tilde{e}_t \colon \tilde{\Gamma}\to\cR^n$ and $\hat{e}_t \colon \tilde{\Gamma}\to\cR^n\times \cR^n$ defined by
$
\tilde{e}_t(\gamma,p)= \gamma(t)
$
and
$
\hat{e}_t(\gamma,p)= (\gamma(t),p(t))
$, for all $(\gamma,p) \in \tilde{\Gamma}$. We denote
\begin{equation*}
\mathcal{P}_{m_0}(\tilde{\Gamma}) = \big\{ \kappa \in \mathcal{P}_1(\tilde{\Gamma}) :\, \tilde{e}_0 \sharp \kappa = m_0 \big\}.
\end{equation*}

Given a distribution $\kappa\in\calp_1(\tilde{\Gamma})$, we consider the cost functional $\tilde{J}^{\kappa}:=J[\mt^\kappa, \tilde{P}^\kappa]$, 
where $\mt^\kappa_t=\tilde{e}_t\sharp\kappa$, for $t\in[0,T]$ 
and where $\tilde{P}^\kappa(t)$ is constructed as an auxiliary function of the
distribution $\hat{e}_t\sharp\kappa$
at time
$t$
in Lemma  \ref{lemma:aux_ex_un}. The well-posedness of $\tilde{J}^\kappa$ is established in Remark \ref{rem:def-P-kappa}.
Once $\tilde{P}^\kappa$ has been defined, we can consider the set of optimal trajectories and associated adjoint states $\tilde{\Gamma}^\kappa[x_0]$ defined by
\begin{equation*}
\tilde{\Gamma}^\kappa[x_0]
= \Big\{ (\bar\gamma,p) \in\tilde{\Gamma}[x_0]\,:\, \bar\gamma\in \Gamma[\mt^\kappa, \tilde{P}^\kappa,x_0]
\text{ and }
p \text{ costate associated with }\bar{\gamma} \Big\}.
\end{equation*}
The precise meaning of ``associated costate" will be given in Definition \ref{def:costate}.

\begin{definition}
  A measure $ \kappa\in \mathcal{P}_{m_0} (\tilde{\Gamma})$
  is an {\it auxiliary MFGC equilibrium} if
\bes
\supp(\kappa)\subset \bigcup_{x\in\supp(m_0)}\tilde{\Gamma}^{\kappa}[x].
\ees
\end{definition}
}\fi

\subsection{Assumptions}

For a given normed vector space $X$, we denote by $\bar{B}_X(R)$ the closed ball of radius $R$ and center 0. When the context is clear, we simply write $\bar{B}(R)$. Given $R>0$, we denote
$V(R)=\text{conv} \big\{ (x,v) : \, |x| \leq R,\, c(x,v) \leq 0 \big\}$. Finally, $\mathbf{1}$ stands for a vector of all ones, of appropriate dimension, and inequality between vectors means component-wise inequality.

We consider the following assumptions:
\begin{itemize}
\item[\hypConv{}] \emph{Convexity assumptions}
\begin{itemize}
\item[\hypConvL{}] There exists $C>0$ such that for all $x \in \cR^n$, the mapping $L(x,\cdot)$ is strongly convex with parameter $1/C$ and 
 for all $(x,v) \in \cR^n \times \cR^m$,
\begin{equation*}
L(x,v) \geq (1/C)|v|^2 - C.
\end{equation*}
\item[\hypConvC{}] For all $x \in \cR^n$ and $i=1,...,n_c$, the mapping $c_i(x,\cdot)$ is convex.
\item[\hypConvPsi{}] The mapping $\psi$ is monotone, i.e.\@ $\langle \psi(y)-\psi(x), y-x \rangle \geq 0$, for all $x$ and $y$ in $\cR^m$.
\end{itemize}
\end{itemize}

\begin{itemize}
\item[\hypReg{}] \emph{Regularity assumptions}
\begin{itemize}
\item[\hypRegL{}] The mappings $L$ and $c$ are twice continuously differentiable.
\item[\hypRegA{}] The mappings $a$, $b$, $g_0$, $g_1$, and $g_2$ are continuously differentiable.
\item[\hypRegF{}] For all $m \in \mathcal{P}_1(\cR^n)$, the mapping $f(\cdot,m)$ is continuously differentiable. The mappings $f$ and $D_x f$ are continuous with respect to both variables.
\item[\hypRegPsi{}] The mapping $\psi$ is continuous.
\end{itemize}
\end{itemize}

\begin{itemize}
\item[\hypBound{}] \emph{Boundedness and growth assumptions}
\begin{itemize}
\item[\hypBoundL{}] Let $R>0$. Then there exists $C(R)>0$ such that, for all $(x,v) \in V(R)$,
\begin{equation*}
\begin{array}{rl}
| D_x L(x,v) | \leq & \! \! C(R)(1 + |v|^2), \\
|D_v L(x,v)| \leq & \! \!  C(R) (1 + |v|).
\end{array}
\end{equation*}
\item[\hypBoundC{}] Let $R>0$. Then there exists $C(R)>0$ such that, for all $(x,v)$ and $(\tilde{x},\tilde{v})$ in $V(R) + \bar{B}(1)$,
\begin{equation*}
\begin{array}{rl}
| D_x c(x,v) | + |D_v c(v,x) | \leq & \! \! C(R), \\
|D_v c(x,v)- D_v c(\tilde{x},\tilde{v})| \leq & \! \! C(R) \big|(x,v)-(\tilde{x},\tilde{v}) \big|.
\end{array}
\end{equation*}
\item[\hypBoundA{}] There exists $C>0$ such that for all $x \in \cR^n$,
\begin{equation*}
|a(x)| \leq C(1 + |x|) \quad \text{and} \quad
|b(x)| \leq C(1 + |x|).
\end{equation*}
\item[\hypBoundSupport{}]
The support $K_0$ of $m_0$ is bounded.
    
\item[\hypBoundF{}] 
There exists $C>0$ such that for all $x_0 \in K_0$,  $m \in C(0,T;\mathcal{P}_1(\cR^n))$ and  $(\gamma,v) \in \mathcal{K}[x_0]$,
\begin{equation*}
\begin{array}{rl}
{\displaystyle \int_0^T } f(\gamma(t),m(t)) \dd t + g_0(\gamma(T), m(T)) \geq & \! \! -C \\[1.5em]
\| D_x f( \gamma(t),m(t)) \|_{L^\infty(0,T;\cR^n)} \leq & \! \! C.
\end{array}
\end{equation*}
\item[\hypBoundPsi{}] The mapping $\psi$ is bounded.
\end{itemize}
\end{itemize}

\begin{itemize}
\item[\hypFeas{}] \emph{Feasibility assumptions}
\begin{itemize}
\item[\hypFeasLocal{}] 
Let $R>0$. Then there exists a constant $C(R)>0$ such that, for all $x \in \bar{B}(R)$, there exists $v \in \bar{B}(C(R))$ satisfying 
$c(x,v) \leq 0$.
\item[\hypFeasGlobal{}] 
There exists $C>0$ such that for all $x_0 \in K_0$,  $m \in C(0,T;\mathcal{P}(\cR^n))$ with $m(0)=m_0$, and for all $P\in L^\infty(0,T;\cR^m)$ satisfying
\begin{equation} \label{eq:bound-P-psi}
\| P \|_{L^\infty(0,T;\cR^m)} \leq \sup_{\theta \in \cR^m} |\psi(\theta)|,
\end{equation}
there exists $(\gamma_0,v_0) \in \mathcal{K}[x]$ such that $J[m,P](\gamma_0,v_0) \leq C$.
\end{itemize}
\end{itemize}

\begin{itemize}
\item[\hypQualif{}] \emph{Qualification assumptions}
\begin{itemize}
\item[\hypQualifEq{}] 
There exists $C>0$ such that for all $x_0 \in K_0$,  $(\gamma,v) \in \mathcal{K}[x_0]$, and  $z_1 \in \cR^{n_{g_1}}$, there exists a pair $(y,w) \in H^1(0,T;\cR^n) \times L^\infty(0,T;\cR^m)$ 
solution of the linearized state equation
\begin{equation} \label{eq:linearized_dyn}
\begin{cases}
\begin{array}{rl}
\dot{y}(t)= & \! \! \! \! \big( Da(\gamma(t)) + {\displaystyle \sum_{i=1}^m} D b_i(\gamma(t)) v_i(t) \big) y(t) + b(\gamma(t)) w(t), \\
y(0)= & \! \! \! \!  0,
\end{array}
\end{cases}
\end{equation}
such that
$
Dg_1(\gamma(T)) y(T)= z_1,
$ 
and
\begin{equation*} 
\| y \|_{H^1(0,T;\cR^n)} \leq C |z_1| \quad \text{and} \quad
\| w \|_{L^\infty(0,T;\cR^n)} \leq C |z_1|.
\end{equation*}
\item[\hypQualifIneq{}] 
There exists $C >0$ such that for all $x_0 \in K_0$ and $(\gamma,v) \in \mathcal{K}[x_0]$, there exists $(y,w) \in H^1(0,T;\cR^n) \times L^\infty(0,T;\cR^m)$ satisfying \eqref{eq:linearized_dyn} such that
\begin{equation*}
\begin{cases}
\begin{array}{rl}
Dg_1(\gamma(T)) y(T)= & \! \! 0, \\
g_2(\gamma(T)) + Dg_2(\gamma(T)) y(T) \leq & \! \! -1/C \mathbf{1}, \\
c(\gamma(t),v(t)) + Dc(\gamma(t),v(t))(y(t),w(t)) \leq & \! \! - 1/C \mathbf{1}.
\end{array}
\end{cases}
\end{equation*}
In addition, 
$
\| y \|_{H^1(0,T;\cR^n)} \leq C 
$
and
$
\| w \|_{L^\infty(0,T;\cR^n)} \leq C.
$
\item[\hypQualifLI{}] There exists $C > 0$ such that for all $(x,v) \in \cR^n \times \cR^m$ satisfying $c(x,v) \leq 0$, and for all $\omega \in \cR^{|I(x,v)|}$,
\begin{equation*}
| D_v c_{I(x,v)}(x,v)^\top \omega | \geq (1/C) |\omega|,
\end{equation*}
where $I(x,v)= \{ i =1,...,n_c :\, c_i(x,u)= 0 \}$.
\item[\hypQualifInward{}]
Given $R>0$, let $x \in \bar{B}(R)$, and let $v \in \cR^m$ be such that $c(x,v) \leq 0$. Then there exist $C>0$, $\varepsilon> 0$, and $w \in \bar{B}(C)$ satisfying
\begin{equation*}
c(x,v) + D_v c(x,v) w \leq - \varepsilon \mathbf{1}.
\end{equation*}
In addition the constants $C>0$ and $\varepsilon>0$ only depend on $R$.

\end{itemize}
\end{itemize}

\begin{remark}
Let us comment on the nature and the motivation of some of the assumptions introduced above.
\begin{enumerate}
\item The first step of our analysis consists in finding a bound in $L^2$ for the optimal controls associated with problem \eqref{eq:control_prob}. This bound must be uniform with respect to $m$, $P$, $x_0$. We proceed with the standard approach from the calculus of variations, which requires:
\begin{itemize}
\item the existence of a feasible pair $(\gamma,v)$ with a uniformly bounded cost: this is ensured by \hypFeas{}-\hypFeasGlobal{}.
\item a lower bound of the cost function $J[m,P](\gamma,v)$ which holds for any feasible pair: this is ensured by \hypConv{}-\hypConvL{} and \hypBound{}-\hypBoundF{}. This also requires a bound on $P$, which is why we also impose that $\psi$ is bounded with Assumption \hypBound{}-\hypBoundPsi{}.
\end{itemize}
\item Assumptions \hypQualif{}-\hypQualifEq{} and \hypQualif{}-\hypQualifIneq{} together are qualification conditions (for the mixed and final constraints) in the form of Mangasarian-Fromovitz qualification conditions (see \cite[Section 2.3.4]{BS13}). They were employed in a similar context in \cite[Section 3]{bonnans2010second}.
\item Assumptions \hypQualif{}-\hypQualifLI{} and \hypQualif{}-\hypQualifInward{} are both qualification conditions for the constraints $c(x,v) \leq 0$ for a fixed value of $x$, this is why those qualification conditions only involve partial derivatives of $c$ with respect to $v$. They are used in particular in Lemma \ref{lemma:reg_metric_mixed} and Lemma \ref{lemma:aux_arg_min_ham}. They respectively take the form of linear independence qualification conditions and inward pointing conditions.
Assumption \hypQualif{}-\hypQualifLI{} was used in \cite[Equation 2.30]{BH09} and
Assumption \hypQualif{}-\hypQualifInward{} was used in \cite[Definition 2.5]{bonnans14} for example, in similar contexts.
\end{enumerate}
\end{remark}

\begin{remark} \label{rem:simplification_assumptions}
\begin{enumerate}
\item As was pointed out above, some of the assumptions are used to derive a priori bounds on the optimal controls associated with problem \eqref{eq:control_prob}.
However, if the set of feasible controls is bounded, these bounds are much easier to obtain and some simplifications can be done.
Assume that there exists a constant $C>0$ such that for any $(x,v) \in \cR^n \times \cR^m$, $c(x,v) \leq 0 \Rightarrow | v | \leq C$. 
Then it is easy to verify that for any $x_0 \in K_0$, for any $(\gamma,v) \in \mathcal{K}[x_0]$, it holds
\begin{equation*}
\| \gamma \|_{L^\infty(0,T;\cR^n} \leq C \quad \text{and} \quad
\| v \|_{L^\infty(0,T;\cR^n} \leq C,
\end{equation*}
increasing if necessary the value of $C$.
In this case, the following simplifications can be considered:
\begin{itemize}
\item Assumptions \hypBound{}-\hypBoundL{} and \hypBound{}-\hypBoundC{} are satisfied.
\item Assumption \hypBound{}-\hypBoundF{} can be ignored. We already have that $\gamma(t)$ takes values in a bounded set.
Moreover, one can require in this assumption that $m(t)$ lies in a set of probability measure with support included into a bounded set; such a set is compact for the topology of $d_1$. Therefore the bounds follow directly from the continuity of $f$, $D_x f$, and $g_0$.
\item It is not necessary to impose that $\psi$ is bounded.
\item Assumption \hypFeas{}-\hypFeasGlobal{} boils down to a feasibility assumption (the bound $J[m,P](\gamma_0,v_0) \leq C$ is then automatically satisfied).
\end{itemize}
\item The verification of Assumption \hypFeas{}-\hypFeasGlobal{} is made easier when $f$ is known to be bounded. Then it suffices to assume that there exists $C>0$ such that for all $x_0 \in K_0$, there exists $(\gamma_0,v_0) \in \mathcal{K}[x_0]$ with $\|v_0\|_{L^2(0,T:\cR^m)}\leq C$. In such a case, it is easy to deduce a bound of $\gamma_0$ in $L^\infty(0,T;\cR^n)$ and finally a bound of $J[m,P](\gamma_0,v_0)$, with the help of Assumption \hypBound{}-\hypBoundL{}.
\item 
In some situations, one can find a convex set $X \subseteq \cR^n$ such that for any $x_0 \in K_0$, for any $(\gamma,v) \in \mathcal{K}[x_0]$, for any $t \in [0,T]$, $\gamma(t) \in X$. In this case, the variable $x$ appearing in Assumptions \hypConv{}-\hypConvL{}, \hypConv{}-\hypConvC{}, \hypFeas{}-\hypFeasLocal{}, \hypQualif{}-\hypQualifLI{}, \hypQualif{}-\hypQualifInward{} can be restricted to $X':=(X+ B_{\cR^n}(\delta))$, where $\delta$ is chosen arbitrarily small. The statements of Lemma \ref{lemma:reg_metric_mixed} and Lemma \ref{lemma:aux_arg_min_ham} remain true for $x \in X$.
\end{enumerate}
\end{remark}


\begin{remark}
For the sake of simplicity in the presentation of this article, we consider time-independent data, but most of the results remain valid if the above assumptions hold uniformly with respect to  time. 
\end{remark}

\section{The optimal control problem}   \label{sec:control-prob}

In this section, we study the optimal control problem \eqref{eq:control_prob} that an individual player aims to solve. Throughout this section, we fix a triplet $(m, P, x_0)\in C(0,T;\calp_1(\cR^n))\times L^\infty(0,T;\cR^m)\times  K_0$ such that \eqref{eq:bound-P-psi} holds.

\subsection{Some technical results}

The next lemma is a metric regularity property, obtained from the Mangasarian-Fromovitz qualification condition \hypQualif{}-\hypQualifInward{}, which implies Robinson's qualification condition (see \cite[Section 2.3.4]{BS13}). Thus the lemma is a particular case of the Robinson-Ursescu stability theorem \cite[Theorem 2.87]{BS13}.

\begin{lemma} \label{lemma:reg_metric_mixed}
Let $R>0$. 
There exist 
  $\delta > 0$ and $C>0$ such that for all $(x,\tilde{x},\tilde{v}) \in \bar{B}(R)^2 \times \cR^m$ such that $c(\tilde{x},\tilde{v}) \leq 0$ and $|x-\tilde{x}| \leq \delta$, there exists $v \in \cR^m$ such that
\begin{equation*}
c(x,v) \leq 0 \quad \text{and} \quad |v-\tilde{v} | \leq C |x-\tilde{x}|.
\end{equation*}
Moreover, for fixed $\tilde{x}$ and $\tilde{v}$, $v$ can be constructed as a continuous function of $x$.
\end{lemma}


\begin{proof}
Let $R>0$.
The constant $\varepsilon$ used below, as well as all constants $C>0$, depend only on $R$.
Let $(x,\tilde{x},\tilde{v}) \in \bar{B}(R)^2 \times \cR^m$ be such that $c(\tilde{x},\tilde{v}) \leq 0$. 
By Assumptions \hypQualif{}-\hypQualifInward{}, there exist $w \in \cR^m$, $C>0$, and $\varepsilon>0$ such that
\begin{equation*}
c(\tilde{x},\tilde{v}) + D_v c(\tilde{x},\tilde{v}) w \leq -\varepsilon \quad \text{and} \quad
|w| \leq C.
\end{equation*}
Let $\theta \in [0,1]$ and let $v_{\theta}= \tilde{v} + \theta w$. We have
\begin{equation} \label{eq:reg_metric_1}
c(x,v_{\theta})
= c(\tilde{x},\tilde{v}) + \theta D_v c(\tilde{x},\tilde{v}) w + a_\theta + b_\theta,
\end{equation}
where
\begin{equation*}
\begin{array}{rl}
a_\theta= & {\displaystyle \int_0^1 } D_x c(\tilde{x} + s(x-\tilde{x}), \tilde{v} + s\theta w) (x-\tilde{x}) \dd s, \\[1em]
b_\theta= & \theta {\displaystyle \int_0^1 } \big[ D_v c(\tilde{x} + s(x-\tilde{x}), \tilde{v} + s\theta w) - D_v(\tilde{x},\tilde{v}) \big] w \dd s.
\end{array}
\end{equation*}
By Assumption \hypBound{}-\hypBoundC{}, we obtain 
\begin{equation*}
|a_\theta| \leq C |x-\tilde{x}| \quad \text{and} \quad
|b_\theta| \leq C \theta \big( |x- \tilde{x}| + \theta \big).
\end{equation*}
It follows from \eqref{eq:reg_metric_1} that
\begin{align}
c(x,v_{\theta})
= \ & (1-\theta) c(\tilde{x},\tilde{v}) + \theta \big[ c(\tilde{x},\tilde{v}) + D_v c(\tilde{x},\tilde{v}) w \big] + a_\theta + b_\theta \notag \\
\leq \ & -\theta \varepsilon + C |x-\tilde{x}| + C \theta^2. \label{eq:reg_metric_easy}
\end{align}
Let us define $\delta = \frac{\varepsilon^2}{4C^2}$ and $\theta = \frac{2C |x-\tilde{x}|}{\varepsilon}$,  
where $C$ is the constant appearing in the right-hand side of \eqref{eq:reg_metric_easy}.
We assume now that $|x-\tilde{x}| \leq \delta$ and we fix 
$
v= v_\theta.
$ 
It remains to verify that $c(x,v) \leq 0$. Note first that 
$
\theta \leq \frac{2C \delta}{\varepsilon} \leq \frac{\varepsilon}{2C}
$, 
by definition of $\delta$. It follows from \eqref{eq:reg_metric_easy} that
\begin{equation*}
c(x,v)
\leq - \theta \varepsilon + \underbrace{(C {\theta})}_{\leq \varepsilon/2} \theta + C |x-\tilde{x} |
\leq - \frac{\theta \varepsilon}{2} + C|x-\tilde{x}|
= 0,
\end{equation*}
which concludes the proof.
\end{proof}

\begin{lemma} \label{lemma:aux_arg_min_ham}
(i) For all $x \in \cR^n$ and for all $r \in \cR^m$, there exists a unique pair $(v,\nu) \in \cR^m \times \cR^{n_c}$  such that the following holds:
\begin{equation} \label{eq:ham_oc}
D_v L(x,v)^\top + r + D_v c(x,v)^\top \nu = 0, \quad
\nu \geq 0, \quad \text{and} \quad
\langle \nu, c(x,v) \rangle = 0.
\end{equation}
We denote it $(v[x,r],\nu[x,r])$. \\
(ii) Let $R>0$. The mapping $(x,r) \in \bar{B}(R) \mapsto (v[x,r],\nu[x,r])$ is Lipschitz continuous.
(iii) There exists $C>0$ such that for all $x \in \cR^n$, for all $r_1$ and $r_2$ in $\cR^m$, it holds
\begin{equation*}
\langle v_2 - v_1, r_2 - r_1 \rangle
+ \frac{1}{C} | v_2 - v_1 |^2 \leq 0,
\end{equation*}
where $v_j= v[x,r_j]$, for $j=1,2$.
\end{lemma}

\begin{proof}
(i) Let $(x,r) \in \cR^n\times\cR^m$. Consider the optimization problem:
\begin{equation} \label{eq:ham_pb}
\inf_{v \in \cR^m} L(x,v) + \langle r,v \rangle, \quad \text{subject to: } c(x,v) \leq 0.
\end{equation}
As a consequence of Assumption \hypConv{}-\hypConvL{}, the above cost function is coercive. By Assumption \hypFeas{}-\hypFeasLocal{}, there exists $v_0$ such that $c(x,v_0) \leq 0$. Therefore, \eqref{eq:ham_pb} possesses a solution $v$. As a consequence of the qualification assumption \hypQualif{}-\hypQualifLI{}, the optimality conditions exactly take the form of \eqref{eq:ham_oc}. This proves the existence part of the first part of the theorem. Now take a pair $(v,\nu)$ satisfying \eqref{eq:ham_oc}. Then, by the strong convexity of $L(x,\cdot)$ and by the convexity of the mappings $c_i(x,\cdot)$, $v$ is the unique solution to \eqref{eq:ham_pb} and $\nu$ is the associated Lagrange multiplier, it is also unique as a consequence of \hypQualif{}-\hypQualifLI{}.

(ii) Let us prove the Lipschitz continuity of $v[\cdot,\cdot]$, $\nu[\cdot,\cdot]$. We mainly rely on results of \cite{BS13}. We first reformulate \eqref{eq:ham_oc} as a generalized equation: given $(x,r)$, the pair $(v,\nu)$ satisfies \eqref{eq:ham_oc} if and only if
\begin{equation} \label{eq:eq_gen}
0 \in \Phi(v,\nu;x,r)+ N(\nu),
\end{equation}
where 
$ 
\Phi(v,\nu;x,r)= \Big( D_v L(x,v)^\top + r + D_v c(x,v)^\top \nu, -c(x,v) \Big) \in \cR^{m+ n_c}
$ 
and
$
N(\nu)=
\big\{ (0,z) \in \cR^{m+n_c} :\, z \leq 0,\, \langle z,\nu \rangle = 0 \big\}
$, if $\nu \geq 0$, 
and $N(\nu)= \emptyset$ otherwise. By \cite[Proposition 5.38]{BS13}, $(\bar{v},\bar{\nu})$ is a strongly regular solution of \eqref{eq:eq_gen} (in the sense of \cite[Definition 5.12]{BS13}). Note that the required sufficient second-order optimality conditions follow from the strong convexity of $L(x,\cdot)$ and the convexity of $c_i(x,\cdot)$. It follows then from \cite[Theorem 5.13]{BS13} that $v[\cdot,\cdot]$ and $\nu[\cdot,\cdot]$ are locally Lipschitz continuous, and therefore Lipschitz continuous on any compact set, as was to be proved.

(iii) Let us subtract equality \eqref{eq:ham_oc}, for $r_1$, from equality \eqref{eq:ham_oc}, for $r_2$, and consider the scalar product of the result with $v_2-v_1$. We obtain
\begin{align*}
& \underbrace{(D_v L(x,v_2)-D_v L(x,v_1))(v_2-v_1)}_{(a)}
+ \langle r_2- r_1,v_2-v_1 \rangle \\
& \qquad - \underbrace{
\langle D_v c(x,v_2)(v_1-v_2), \nu_2 \rangle}_{(b)}
- \underbrace{\langle D_v c(x,v_1)(v_2-v_1), \nu_1 \rangle}_{(c)}= 0,
\end{align*}
where $\nu_j= \nu[x,r_j]$, for $j=1,2$.
To conclude the proof, we just need to bound from below the term $(a)$ and to bound from above $(b)$ and $(c)$. By Assumption \hypConv{}-\hypConvL{}, we have $(a) \geq \frac{1}{C} | v_2-v_1 |^2$, for some constant $C>0$ independent of $x$, $r_1$, and $r_2$.
Using the complementarity condition, the convexity of $c$ with respect to its second variable (Assumption \hypConv{}-\hypConvC{}), and the nonnegativity of $\nu_2$, we further obtain that
\begin{align*}
(b) = \langle c(x,v_2) + D_v c(x,v_2)(v_1-v_2), \nu_2 \rangle
\leq \langle c(x,v_1), \nu_2 \rangle \leq 0.
\end{align*}
Similarly, $(c) \leq 0$. This concludes the proof.
\end{proof}

\begin{remark}
The twice differentiability of $L$ and $c$, required in Assumption \hypReg{}-\hypRegL{} is only used for the application of \cite[Proposition 5.38]{BS13} in the proof of Lemma \ref{lemma:aux_arg_min_ham}. It is sufficient to assume that $L$ and $c$ are continuously differentiable if $c$ does not depend on $x$ (i.e.\@ if we just have control constraints instead of mixed state-control constraints). In that case, the Lipschitz continuity is deduced from \cite[Proposition 4.32]{BS13}.
\end{remark}

\subsection{Estimates for the optimal solutions} 

The goal of this section is to derive some a priori bounds for solutions $(\bar{\gamma},\bar{v})$ to the optimal control problem \eqref{eq:control_prob} and for the associated costate and Lagrange multipliers. They will be crucial for the construction of an appropriate set of probability measures on state-costate trajectories.
We follow a rather standard methodology. The coercivity of $L$, together with other feasibility and bound conditions allows to show the existence of a solution and to derive a bound of $\bar{v}$ in $L^2(0,T;\cR^m)$. Then we provide first-order necessary optimality conditions and a bound on the associated costate $p$, with the help of the qualification conditions. We finally obtain bounds of $\bar{\gamma}$ and $p$ in $W^{1,\infty}(0,T;\cR^n)$ and $\vb$ in $L^\infty(0,T;\cR^m)$.

We recall that throughout this section the triplet $(m,P,x_0)$
is fixed and satisfies \eqref{eq:bound-P-psi}. Note that all constants $C$ used in this section are independent of $(m,P,x_0)$.

\begin{proposition} \label{prop:bound_control_pb_gamma}
The optimal control problem \eqref{eq:control_prob} has (at least) one solution.
There exist two constants 
 $M_1 > 0$ and $C>0$, independent of $m$, $P$, and $x_0$, such that for all solutions $(\bar{\gamma},\bar{v})$ to \eqref{eq:control_prob},
\begin{equation} \label{eq:bound_control_pb}
\| \bar{\gamma} \|_{L^\infty(0,T;\cR^n)} \leq M_1 \quad \text{and} \quad
\| \bar{v} \|_{L^2(0,T;\cR^m)} \leq C.
\end{equation}
\end{proposition}

\begin{proof}
The constants $C>0$ used in this proof only depend on the data of the problem.
Let $(\gamma_0,v_0) \in \mathcal{K}[x_0]$ satisfy Assumption \hypFeas{}-\hypFeasGlobal{}. Let $(\gamma_k,v_k)_{k \in \mathbb{N}}$ be a minimizing sequence. Without loss of generality, we can assume that
\begin{equation*}
J[m,P](\gamma_k,v_k) \leq J[m,P](\gamma_0,v_0), \quad
\forall k \in \mathbb{N}.
\end{equation*}
Using Assumption \hypConv{}-\hypConvL{}, the boundedness of $P$, and Assumption \hypBound{}-\hypBoundF{}, we deduce that
\begin{align*}
C \geq \ & J[m,P](\gamma_0,v_0) \geq J[m,P](\gamma_k,v_k) \notag \\
\geq \ & \frac{1}{C} \| v_k \|_{L^2(0,T)}^2
- C \| v_k \|_{L^1(0,T)} - C \geq \frac{1}{C} \| v_k \|_{L^2(0,T)}^2 - C,
\end{align*}
for some independent constants $C$.
It follows that $v_k$ is bounded in $L^2(0,T;\cR^m)$. By Gr\"onwall's lemma and Assumption \hypBound{}-\hypBoundA, there exists a constant $C>0$ such that
$
\| \gamma_k \|_{L^\infty(0,T;\cR^n)} \leq C.
$
The state equation further implies that
$\| \gamma_k \|_{H^1(0,T;\cR^n)} \leq C$.
Extracting a subsequence if necessary,  there exist $(\bar{\gamma},\bar{v}) \in H^1(0,T;\cR^n) \times L^2(0,T;\cR^m)$ and a $C>0$ such that
\begin{equation*} 
\| \bar{\gamma} \|_{H^1(0,T;\cR^n)} \leq C \quad \text{and} \quad \| \bar{v} \|_{L^2(0,T;\cR^m)} \leq C
\end{equation*}
and such that $(\gamma_k,v_k) \rightharpoonup (\bar{\gamma},\bar{v})$ for the weak topology of $H^1(0,T;\cR^n) \times L^2(0,T;\cR^m)$. Since $H^1(0,T;\cR^n)$ is compactly embedded in $L^\infty(0,T;\cR^n)$, we deduce that $\gamma_k$ converges uniformly to $\bar{\gamma}$.

Let us prove that $c(\bar{\gamma}(t),\bar{v}(t)) \leq 0$ for a.e.\@ $t \in (0,T)$. Let $\varphi \in L^\infty(0,T;\cR^{n_c})$ be such that $\varphi(t) \geq 0$ for a.e.\@ $t \in (0,T)$.
We have
\begin{equation*}
\int_0^T \langle \varphi(t), c(\bar{\gamma}(t),\bar{v}(t)) \rangle \dd t
= a_k + b_k + \int_0^T \langle \varphi(t),c(\gamma_k(t),v_k(t)) \rangle \dd t
\leq a_k + b_k,
\end{equation*}
where, skipping the time arguments
\begin{equation*}
a_k =  {\displaystyle \int_0^T} \big\langle \varphi, c(\bar{\gamma},\bar{v})-c(\bar{\gamma},v_k) \big\rangle \dd t\quad \text{and}\quad
b_k =  {\displaystyle \int_0^T} \big\langle \varphi, c(\bar{\gamma},v_k) - c(\gamma_k,v_k) \big\rangle \dd t.
\end{equation*}
Note that all these integrals are well-defined as a consequence of Assumption \hypBound{}-\hypBoundC{}.
Also by Assumption \hypBound{}-\hypBoundC{}, we easily verify that $D_v c(\bar{\gamma}(\cdot),\bar{v}(\cdot)) \in L^2(0,T;\cR^{n_c \times m})$. Therefore, by the convexity of the mappings $c_i(x,\cdot)$ in Assumption \hypConv{}-\hypConvC{},
\begin{equation*}
a_k \leq
\int_0^T \langle \varphi(t), D_v c(\bar{\gamma}(t),\bar{v}(t))(\bar{v}(t)-v_k(t)) \rangle \dd t
\underset{k \to \infty}{\longrightarrow} 0.
\end{equation*}
By Assumption \hypBound{}-\hypBoundC{}, we also have
\begin{equation*}
|b_k| \leq C \| \varphi \|_{L^\infty(0,T;\cR^{n_c})} \| \gamma_k - \bar{\gamma} \|_{L^\infty(0,T;\cR^n)}
\underset{k \to \infty}{\longrightarrow} 0.
\end{equation*}
It follows that for all $\varphi \geq 0$, $\int_0^T \langle \varphi, c(\bar{\gamma}(t),\bar{v}(t)) \rangle \dd t \leq 0$. Therefore, $c(\bar{\gamma}(t),\bar{v}(t)) \leq 0$, for a.e.\@ $t \in (0,T)$.
With similar arguments, we prove that $(\bar{\gamma},\bar{v})$ is feasible and that
\begin{equation*}
J[m,P](\bar{\gamma},\bar{v}) \leq \lim_{k \to \infty} J[m,P](\gamma_k,v_k),
\end{equation*}
which concludes the proof of optimality of $(\bar{\gamma},\bar{v})$.
Repeating the above arguments, we show that any solution to \eqref{eq:control_prob} satisfies the bound \eqref{eq:bound_control_pb}.
\end{proof}

We next state optimality conditions for the optimal control problem. The proof of the following proposition is deferred to the appendix in Section \ref{section:proof_oc}. In the rest of the section, we write $c[t]$ instead of $c(\bar{\gamma}(t),\bar{v}(t))$ (for a specified pair $(\bar{\gamma},\bar{v})$). We use the same convention for $a$, $b$, $g_0$, $g_1$, and $g_2$.

\begin{proposition} \label{prop:opti_cond}
Let $(\bar{\gamma},\bar{v})$ be a solution to \eqref{eq:control_prob}.
There exists a quintuplet
\begin{equation*}
(p,\lambda_0,\lambda_1,\lambda_2,\nu) \in W^{1,2}(0,T;\cR^n) \times \cR \times \cR^{n_{g_1}} \times \cR^{n_{g_2}} \times L^\infty(0,T;\cR^{n_c})
\end{equation*}
such that $(p,\lambda_0) \neq (0,0)$ and such that the adjoint equation
\begin{equation} \label{eq:opti_cond_oc1}
\left\{
\begin{array}{rl}
p(T)^\top \hspace{-0.1cm} = & \hspace{-0.2cm} \! \lambda_0 D g_0[T] + \lambda_1^\top D g_1[T] + \lambda_2^\top D g_2[T] \\[6pt]
-\dot{p}(t)^\top
\hspace{-0.1cm} = & \hspace{-0.2cm} \! \lambda_0 D_x L[t] + \lambda_0 D_x f[t]
+ p(t)^\top \big(D a [t] + \sum_{i=1}^m D b_i[t] v_i(t) \big)
+ \nu(t)^\top D_x c[t],
\end{array}
\right.
\end{equation}
the stationary condition
\begin{equation} \label{eq:opti_cond_oc2}
\lambda_0 D_v L [t] + \lambda_0 P(t)^\top + p(t)^\top b[t] + \nu(t)^\top D_v c[t]= 0,
\end{equation}
and the following sign and complementarity conditions
\begin{equation} \label{eq:opti_cond_oc3}
\begin{cases}
\begin{array}{ll}
\lambda_0 \geq 0, & \\
\lambda_2 \geq 0, \quad &
\langle \lambda_2, g_2(\bar{\gamma}(T)) \rangle = 0, \\
\nu(t) \geq 0, &
\langle c(\bar{\gamma}(t),\bar{v}(t)), \nu(t) \rangle= 0, \quad
\text{for a.e. $t \in (0,T)$}
\end{array}
\end{cases}
\end{equation}
are satisfied.
Moreover, if $\lambda_0\neq 0$, then $p\in W^{1,\infty}(0,T;\cR^n)$.
\end{proposition}

The goal of the last two results in this subsection is to obtain uniform bounds for the optimal solutions and their associated multipliers.

\begin{proposition} \label{prop_bound_p}
Let $(\bar{\gamma},\bar{v})$ be a solution to \eqref{eq:control_prob}.
There exists a quintuplet
$(p,\lambda_0,\lambda_1,\lambda_2,\nu)$ satisfying the optimality conditions of the above proposition and such that $\lambda_0= 1$. Moreover, for such a quintuplet, we have
\begin{equation*}
\| p \|_{L^\infty(0,T;\cR^n)} \leq M_2
\end{equation*}
for some constant $M_2$ independent of $(\bar{\gamma},\bar{v})$ and $(p,\lambda_0,\lambda_1,\lambda_2,\nu)$.
\end{proposition}

\begin{proof}
The proof essentially relies on the qualification conditions \hypQualif{}-\hypQualifEq{} and \hypQualif{}-\hypQualifIneq{}. All constants $C$ used in the proof are independent of $(p,\lambda_0,\lambda_1,\lambda_2,\nu)$ and $(\bar{\gamma},\bar{v})$.
Let $(y,w)$ satisfy the linearized equation \eqref{eq:linearized_dyn} (for $(\gamma,v)= (\bar{\gamma},\bar{v})$).
By integration by parts we have
\begin{align*}
\langle p(T),y(T) \rangle
= \ & \int_0^T \langle \dot{p}(t),y(t) \rangle + \langle p(t), \dot{y}(t) \rangle \dd t \\
= \ & - \lambda_0 \int_0^T (D_x L[t] + D_x f[t]) y(t) \dd t \\
&  - \int_0^T p(t)^\top \big( D a[t] + {\textstyle \sum_{i=1}^m} D b_i[t] \bar{v}_i(t) \big) y(t) \dd t
- \int_0^T \nu(t)^\top D_x c[t] y(t) \dd t \\
& + \int_0^T p(t)^\top \big( Da[t] + {\textstyle \sum_{i=1}^m} D b_i[t] \bar{v}_i(t) \big) y(t) \dd t
+ \int_0^T p(t)^\top b[t] w(t) \dd t.
\end{align*}
The second and the fourth integral cancel out. Injecting the optimality condition \eqref{eq:opti_cond_oc2} in the 
last integral, we obtain:
\begin{align}    \label{eq:ipp}
\langle p(T),y(T) \rangle
= \ & - \lambda_0 \int_0^T \Big( D L[t](y(t),w(t)) + \langle P(t),w(t) \rangle + D_x f[t] y(t) \Big) \dd t \notag \\
& \quad -\int_0^T  \nu(t)^\top Dc[t](y(t),w(t)) \dd t.
\end{align}
The main feature of this formula is that the right-hand side is independent of $p$.
Let $(y,w)$ satisfy Assumption \hypQualif{}-\hypQualifIneq{}. 
  By \hypBound{}-\hypBoundL{}, \hypBound{}-\hypBoundPsi{} and \hypBound{}-\hypBoundF{} we have
\bes 
\ba{l}
\int_0^T \left(D L[t](y(t),w(t)) + \langle P(t),w(t) \rangle + D_x f[t] y(t)\right) \dd t\\[5pt]
\hspace{2cm}\leq C \|y\|_\infty\int_0^T(1+|\vb(t)|^2)\dd t+ C\|w\|_\infty\int_0^T(1+|\vb(t)|)\dd t \leq C,
\ea 
\ees
where the last inequality holds by Proposition \ref{prop:bound_control_pb_gamma} and \hypQualif{}-\hypQualifIneq{}.
By \hypQualif{}-\hypQualifIneq{} and the complementarity conditions \eqref{eq:opti_cond_oc3} we obtain
\begin{align*}
\int_0^T  \nu(t)^\top Dc[t](y(t),w(t)) \dd t=& \ \int_0^T  \nu(t)^\top \left(c[t]+Dc[t](y(t),w(t))\right) \dd t\\
\leq & \ -\frac{1}{C}\|\nu\|_{L^1(0,T;\cR^{n_c})}.
\end{align*}
Therefore,
\begin{equation} \label{eq:prod_scal1}
\langle p(T),y(T) \rangle \geq -C \lambda_0 + \frac{1}{C} \| \nu \|_{L^1(0,T;\cR^{n_c})}.
\end{equation}
Moreover, we deduce from the terminal condition for $p$ that
\begin{align}
\langle p(T),y(T) \rangle
= & \ \lambda_0 Dg_0[T] y(T) + \langle \lambda_1, Dg_1[T] y(T) \rangle
+ \langle \lambda_2, Dg_2[T] y(T) \rangle \notag \\
= & \ \lambda_0 Dg_0[T] y(T) + \langle \lambda_2, g_2[T] + Dg_2[T] y(T) \rangle \notag \\
\leq & \ C \lambda_0- \frac{1}{C} | \lambda_2 |. \label{eq:prod_scal2}
\end{align}

The last inequality holds by \hypReg{}-\hypRegA{}, Proposition \ref{prop:bound_control_pb_gamma} and   \hypQualif{}-\hypQualifIneq{}.
It follows from \eqref{eq:prod_scal1} and \eqref{eq:prod_scal2} that
\begin{equation} \label{eq:bound_lambda_2}
|\lambda_2| \leq C \lambda_0 \quad \text{and} \quad
\| \nu \|_{L^1(0,T;\cR^{n_c})} \leq C \lambda_0.
\end{equation}
Now, let us consider $(y,w)$ satisfying \hypQualif{}-\hypQualifEq{} with $z_1= \lambda_1/ | \lambda_1 |$.
We have
\begin{equation*}
\| y \|_{L^\infty(0,T;\cR^n)} \leq C \quad \text{and} \quad \| w \|_{L^\infty(0,T;\cR^n)} \leq C.
\end{equation*}
Since $Dc[t]$ is bounded in $L^\infty(0,T;\cR^{n_c \times (n+m)})$ (by Assumption \hypBound{}-\hypBoundC{}), we have
\begin{equation} \label{eq:estimDc}
\| Dc[\cdot](y(\cdot),w(\cdot)) \|_{L^\infty(0,T;\cR^{n_c})} \leq C.
\end{equation}
Formula \eqref{eq:ipp}, together with the bound on $\| \nu \|_{L^1(0,T;\cR^{n_c})}$ and \eqref{eq:estimDc} yields
\begin{equation} \label{eq:bound_lambda_1a}
\langle p(T), y(T) \rangle \leq  C \lambda_0.
\end{equation} 
It follows from the terminal condition and  the estimate on $|\lambda_2|$ that
\begin{align}
\langle p(T), y(T) \rangle
= \ &\lambda_0 Dg_0[T] y(T)
+ \langle \lambda_1, Dg_1[T] y(T) \rangle
+ \langle \lambda_2, Dg_2[T] y(T) \rangle \notag  \\
\geq \ & -C \lambda_0 + | \lambda_1 |. \label{eq:bound_lambda_1b}
\end{align}
Combining \eqref{eq:bound_lambda_1a} and \eqref{eq:bound_lambda_1b}, we deduce that
\begin{equation} \label{eq:bound_lambda_1}
| \lambda_1 | \leq C \lambda_0.
\end{equation}
If $\lambda_0= 0$, then $\lambda_1= 0$, $\lambda_2=0$, and $\nu= 0$. Thus $p(T)=0$ and $\dot{p}(t)=0$ a.e.\@ and therefore $p=0$, in contradiction with $(p,\lambda_0) \neq (0,0)$. We deduce that $\lambda_0 > 0$. The optimality conditions being invariant by multiplication of a positive constant, we deduce the existence of a quintuplet satisfying \eqref{eq:opti_cond_oc1}-\eqref{eq:opti_cond_oc2}-\eqref{eq:opti_cond_oc3} and $\lambda_0=1$. Bounds of $|\lambda_1|$, $|\lambda_2|$, and $\| \nu \|_{L^1(0,T;\cR^{n_c})}$ directly follow from \eqref{eq:bound_lambda_2} and \eqref{eq:bound_lambda_1}.
Then we obtain a bound of $|p(T)|$ and finally a bound of $\| p \|_{L^\infty(0,T;\cR^n)}$ with Gr{\"o}nwall's lemma.
\end{proof}

\begin{definition} \label{def:costate}
Given a solution $(\bar{\gamma},\bar{v})$ to \eqref{eq:control_prob}, we call associated costate any $p$ for which there exists $(\lambda_0,\lambda_1,\lambda_2,\nu)$ such that the optimality conditions in Proposition \ref{prop:opti_cond} hold true and $\lambda_0= 1$.
\end{definition}

In order to obtain more regularity on $(\bar{\gamma},\bar{v})$, we need to express the optimal control as an auxiliary function of the state and costate, which is deduced from 
 Lemma \ref{lemma:aux_arg_min_ham}.

\begin{lemma}  \label{lem:bound-deriv-gamma-p} Let $(\bar \gamma, \vb)$ and $(p,\lambda_0,\lambda_1,\lambda_2,\nu)$ 
be as in Proposition \ref{prop_bound_p}.
There exists $C>0$ independent of $(\bar\gamma, \vb)$ and $(p, \lambda_0, \lambda_1, \lambda_2,\nu)$ such that
\bes 
\| \vb \|_{L^\infty(0,T;\cR^n)} \leq  C,\;\;\; \| \nu \|_{L^\infty(0,T;\cR^n)} \leq  C.
\ees
In addition, there exist constants $M_3>0$ and $M_4>0$,  such that
\begin{equation} \label{eq:bound_ultimate}
\| \dot{\bar{\gamma}} \|_{L^\infty(0,T;\cR^n)} \leq  M_3,\;\;\; \| \dot{p} \|_{L^\infty(0,T;\cR^n)} \leq M_4.
\end{equation}
\end{lemma}

\begin{proof}
It follows from the optimality condition \eqref{eq:opti_cond_oc2} and Lemma \ref{lemma:aux_arg_min_ham} that
\begin{equation} \label{eq:reg_control}
\begin{array}{rl}
\bar{v}(t)
= & v[\bar{\gamma}(t),P(t)+ b(\bar{\gamma}(t))^\top p(t)], \\
\nu(t) = & \nu[\bar{\gamma}(t),P(t)+ b(\bar{\gamma}(t))^\top p(t)].
\end{array}
\end{equation}
Lemma \ref{lemma:aux_arg_min_ham} further implies that $\| \bar{v} \|_{L^\infty(0,T;\cR^m)} \leq C$ and 
$\| \nu \|_{L^\infty(0,T;\cR^{n_c})} \leq C$. The estimates \eqref{eq:bound_ultimate} follow.
\end{proof}

\section{Existence of MFGC equilibria}   \label{sec:existence}

In this section, we prove the main result of the paper.
We first construct the auxiliary function announced in the introduction. Then, applying Kakutani's fixed point theorem, we prove the existence of an auxiliary MFGC equilibrium  (defined in Subsection \ref{auxMFGCequ}), which will imply the existence of a Lagrangian one.

\subsection{Auxiliary function}  \label{subsec:aux-func}

We set $B= \bar{B}(M_1) \times \bar{B}(M_2) \subset \cR^n \times \cR^n$, where $M_1$ and $M_2$ are given by Lemma \ref{prop:bound_control_pb_gamma} and Proposition \ref{prop_bound_p}, respectively.

\begin{lemma} \label{lemma:aux_ex_un}
(i) Let $\mu \in \mathcal{P}(\cR^n \times \cR^n)$ be such that $\supp(\mu) \subseteq B$. There exists a unique $P \in \cR^m$ such that
\begin{equation} \label{eq:cond_auxil}
P= \psi \Big( \int_B v[x, P+ b(x)^\top q] \dd \mu(x,q) \Big),
\end{equation}
where $v[\cdot,\cdot]$ is the mapping introduced in Lemma \ref{lemma:aux_arg_min_ham}.
In the sequel, the unique solution to \eqref{eq:cond_auxil} is denoted by $\mathbf{P}[\mu]$.

(ii) The mapping $\mu \in \mathcal{P}(B) \mapsto \mathbf{P}[\mu]$ is uniformly continuous.
\end{lemma}

The mapping $\mathbf{P}$ is the auxiliary function which will allow us later to define a new notion of equilibrium. Let us motivate its introduction. In the above lemma, the probability measure $\mu$ represents the distribution of the agents with respect to their state and costate at a given time $t$. At equilibrium, an agent with state $x$ and costate $q$ utilizes the control $v[x,P(t)+b(x)^\top q]$, by Pontryagin's principle. Therefore the price relation \eqref{eq:def-m-P-eta-Lagrangian} leads to the fixed point equation $P(t)= \psi(\int_B v[x,P(t)+b(x)^\top q] \dd \mu(x,q))$ introduced above.

\begin{proof}[Proof of Lemma \ref{lemma:aux_ex_un}]
Let us first prove the existence of a solution. Let $C>0$ denote a bound of $|\psi|$ (Assumption \hypBound{}-\hypBoundPsi{}). Consider the map
\begin{equation*}
\chi \colon P \in \bar{B}_{\cR^m}(C) \mapsto \psi \Big( \int_B v[x, P + b(x)^\top q] \dd \mu(x,q) \Big) \in \bar{B}_{\cR^m}(C).
\end{equation*}
By Lemma \ref{lemma:aux_arg_min_ham}, the mapping $v[\cdot,\cdot]$ is continuous. Therefore, by the Schauder fixpoint theorem, there exists $P \in \bar{B}_{\cR^m}(C)$ such that $P= \chi(P)$, which proves the existence of a solution to \eqref{eq:cond_auxil}.

Let us prove next the uniqueness and the uniform continuity.
Let $\mu_1$ and $\mu_2$ be in $\mathcal{P}(B)$. Let $P_1$ and $P_2$ denote two solutions of \eqref{eq:cond_auxil}, for $\mu= \mu_1$ and $\mu=\mu_2$, respectively. For $j=1,2$, consider the maps
\begin{equation*}
v_j \colon (x,q) \in B \mapsto v[x, P_j + b(x)^\top q ].
\end{equation*}
Note that by construction, $P_j = \psi \big( \int_B v_j \dd \mu_j \big)$.
Let us first note that there exists a constant $C>0$, independent of $\mu_1$, $\mu_2$, $v_1$, and $v_2$ such that
\begin{equation*}
\| v_j \|_{L^\infty(B)}  \leq C \quad \text{and} \quad
\text{$v_j$ is $C$-Lipschitz}.
\end{equation*}
This is a consequence of Assumption \hypReg{}-\hypRegA{} and Lemma \ref{lemma:aux_arg_min_ham}. 
This implies, together with the monotonicity of $\psi$ (Assumption \hypConv{}-\hypConvPsi{}) that
\begin{align*}
0 \leq \ & \langle \psi({\textstyle \int_B} v_2 \dd \mu_2 ) - \psi({\textstyle \int_B} v_1 \dd \mu_2) , {\textstyle \int_B} (v_2 - v_1) \dd \mu_2 \rangle \\
= \ & \underbrace{\langle P_2 - P_1, {\textstyle \int_B} (v_2- v_1) \dd \mu_2 \rangle}_{=: (a)}
+
\underbrace{\langle \psi({\textstyle \int_B} v_1 \dd \mu_1) - \psi({\textstyle \int_B} v_1 \dd \mu_2), {\textstyle \int_B} (v_2- v_1) \dd \mu_2 \rangle}_{=: (b)}.
\end{align*}
Lemma \ref{lemma:aux_arg_min_ham} yields
\begin{align*}
(a) \leq \ &
\int_B \langle v_2(x,q)-v_1(x,q), (P_2 + b(x)^\top q)-(P_1 + b(x)^\top q) \rangle\dd \mu_2(x,q) \\
\leq \ & - \frac{1}{C} \int_B |v_2(x,q)-v_1(x,q)|^2 \dd \mu_2(x,q).
\end{align*}
Let $C$ denote a bound of $\| v_1 \|_{L^\infty(B)}$. Since $\psi$ is continuous, it is uniformly continuous on $B$. Therefore, there exists a function $\omega \colon [0,\infty) \rightarrow [0,\infty)$ such that for all $x$ and $y$ in $\bar{B}_{\cR^m}(C)$, $|\psi(y)-\psi(x)| \leq \omega(|y-x|)$, such that $\omega(0)=0$ and such that $\omega$ is right-continuous at 0.
We have
\begin{equation*}
(b) \leq \omega \big( | {\textstyle \int_B} v_1 \dd \mu_1-  {\textstyle \int_B} v_1 \dd \mu_2 | \big)  \Big|  \int_B (v_2 -v_1) \dd \mu_2 \Big|. 
\end{equation*}
Using further the Lipschitz continuity of $v_1$ and Cauchy-Schwarz inequality, we deduce that
\begin{equation*}
(b) \leq \omega( C d_1(\mu_1,\mu_2)) \, \Big( \int_B |v_2-v_1|^2 \dd \mu_2 \Big)^{1/2}.
\end{equation*}
Since $0 \leq (a) + (b)$, we deduce that
\begin{equation*}
\int_B |v_2- v_1 |^2 \dd \mu_2
\leq \Big( C \omega( C d_1(\mu_1,\mu_2)) \Big)^2 =: \tilde{\omega}(d_1(\mu_1,\mu_2)).
\end{equation*}
Finally, we have
\begin{align*}
|P_2-P_1 |
\leq \ & \omega\big( |{\textstyle \int_B} v_2 \dd \mu_2 - {\textstyle \int_B}  v_1 \dd \mu_1 | \big) \\
\leq \ & \omega\big( |{\textstyle \int_B} v_2-v_1 \dd \mu_2| + |{\textstyle \int_B} v_1 \dd \mu_2- {\textstyle \int_B} v_1 \dd \mu_1 | \big) \\
\leq \ & \omega( \tilde{\omega}(d_1(\mu_1,\mu_2) + C d_1(\mu_1,\mu_2)).
\end{align*}
If $\mu_1= \mu_2$, then $P_1= P_2$. The uniqueness of the solution to \eqref{eq:cond_auxil} follows. The uniform continuity of $P$ also follows, which concludes the proof.
\end{proof}

\color{black}

\subsection{Auxiliary MFGC equilibria}\label{auxMFGCequ}

In order to analyze the existence of Lagrangian MFGC equilibria, 
we propose here a new notion of equilibrium, that we call auxiliary equilibrium. 
 We set
\begin{equation*}
\tilde{\Gamma}= H^1(0,T;\cR^n) \times H^1(0,T;\cR^n).
\end{equation*}
We equip $\tilde{\Gamma}$ with the supremum norm, defined by 
$ 
\max( \| \gamma \|_\infty, \| p \|_\infty)
$ 
for a given pair $(\gamma,p) \in \tilde{\Gamma}$. We denote it (by extension) $\| (\gamma,p) \|_\infty$.
For any $x_0 \in \cR^n$, we define
\begin{equation*}
\tilde{\Gamma}[x_0]= \big\{ (\gamma,p) \in \tilde{\Gamma} \,:\, \gamma(0)= x_0 \big\}.
\end{equation*}
Given $t \in [0,T]$, we consider the mappings $\tilde{e}_t \colon \tilde{\Gamma}\to\cR^n$ and $\hat{e}_t \colon \tilde{\Gamma}\to\cR^n\times \cR^n$ defined by
$
\tilde{e}_t(\gamma,p)= \gamma(t)
$
and
$
\hat{e}_t(\gamma,p)= (\gamma(t),p(t))
$, for all $(\gamma,p) \in \tilde{\Gamma}$. We denote
\begin{equation*}
\mathcal{P}_{m_0}(\tilde{\Gamma}) = \big\{ \kappa \in \mathcal{P}_1(\tilde{\Gamma}) :\, \tilde{e}_0 \sharp \kappa = m_0 \big\}.
\end{equation*}

We consider the following compact subset of $\tilde{\Gamma}$, 
\be
\tilde{\Gamma}_B:=\{(\gamma, p)\in \tilde{\Gamma} : \|\gamma\|_\infty\leq M_1, \;\|p\|_\infty\leq M_2,\; \|\dot{\gamma}\|_2\leq T^\frac{1}{2}M_3, \;\|\dot{p}\|_2\leq T^\frac{1}{2}M_4\},   \label{eq:Gamma_B}
\ee
where $M_1$, $M_2$, $M_3$ and $M_4$ were introduced in Proposition \ref{prop:bound_control_pb_gamma}, Proposition \ref{prop_bound_p} and Lemma \ref{lem:bound-deriv-gamma-p}.

Given a distribution $\kappa\in\calp_1(\tilde{\Gamma})$ with $\supp(\kappa)\subset \tilde{\Gamma}_B$, we set for $t\in [0,T]$
$$
\mt^\kappa_t=\tilde{e}_t\sharp\kappa \in\calp(\cR^n), \quad \mu^\kappa_t=\hat{e}_t\sharp\kappa \in\calp(\cR^n\times\cR^n).
$$

\begin{lemma}   \label{lem:m-mu-hol}
Let $\kappa \in \calp_{m_0}(\tilde{\Gamma})$ with $\supp(\kappa)\subset \tilde{\Gamma}_B$. Then $\mt^\kappa_t$ and $\mu^\kappa_t$ are $\frac{1}{2}$-H\"older continuous w.r.t.\@ $t \in [0,T]$. 
\end{lemma}

\begin{proof} Recalling that $B=\overline{B}(M_1)\times \overline{B}(M_2)\subset \cR^n\times\cR^n$, since $\supp(\kappa)\subset \tilde{\Gamma}_B$, we obtain $\supp(\mu^\kappa_t)\subset B$ and $\supp(\mt^\kappa_t)\subset \bar{B}(M_1)$ for all $t\in [0,T]$. 

For all $s,t\in [0,T]$ we have
\begin{align*}
d_1(\mu^\kappa_t, \mu^\kappa_s)=&\ \sup_{\varphi\in {\rm Lip}_1(\cR^n\times\cR^n)}  \int_B\varphi(x,q)(\dd \mu^\kappa_t-\dd \mu^\kappa_s)(x,q)\\[6pt]
\;=&\ \sup_{\varphi\in {\rm Lip}_1(\cR^n\times\cR^n)}\int_\Gamma[\varphi(\gamma(t), p(t))-\varphi(\gamma(s), p(s))]\dd \kappa(\gamma, p)\\[6pt]
\;\leq & \ \int_\Gamma\max\{|\gamma(t)-\gamma(s)|,| p(t)- p(s)|\} \dd \kappa(\gamma, p)\\[6pt]
\;\leq& \ T^\frac{1}{2}\max\{M_3,M_4\}|t-s|^\frac{1}{2}.
\end{align*}
The last inequality holds by the assumption $\supp(\kappa)\subset \tilde{\Gamma}_B$. Similarly the result follows for $\mt^\kappa_t$.
\end{proof}

Given $\kappa\in \calp(\tilde{\Gamma}_B)$, by  the above lemma, we obtain $\mt^\kappa\in C([0,T];\calp_1(\cR^n))$. 
Setting $\tilde{P}^\kappa\in L^\infty(0,T:\cR^m)$ given by
\be\label{eq:def-tilde-P-kappa}
\tilde{P}^\kappa(t)={\bf P}[\mu^\kappa_t],
\ee 
where $\mathbf{P}$ is defined in Lemma \ref{lemma:aux_ex_un}, by Lemma \ref{lemma:aux_ex_un} and Lemma \ref{lem:m-mu-hol}, we obtain $\tilde{P}^\kappa\in C(0,T;\cR^m)$.  
 Defining the 
functional $\tilde{J}^{\kappa}:=J[\mt^\kappa, \tilde{P}^\kappa]$,  
we can consider the set of optimal trajectories and associated adjoint states $\tilde{\Gamma}^\kappa[x_0]$ 
given by
\begin{equation*}
\tilde{\Gamma}^\kappa[x_0]
= \Big\{ (\bar\gamma,p) \in\tilde{\Gamma}[x_0]\,:\, \bar\gamma\in \Gamma[\mt^\kappa, \tilde{P}^\kappa,x_0]
\text{ and }
p \text{ is an associated costate with }\bar{\gamma} \Big\},
\end{equation*}
where the meaning of ``associated costate" is given in Definition \ref{def:costate} and $\Gamma[\mt^\kappa, \tilde{P}^\kappa,x_0]$
 was defined in \eqref{def-gam-m-p0-x}.

\begin{definition}
  A measure $ \kappa\in \mathcal{P}_{m_0} (\tilde{\Gamma})$
  is an {\it auxiliary MFGC equilibrium} if
\bes
\supp(\kappa)\subset \bigcup_{x\in\supp(m_0)}\tilde{\Gamma}^{\kappa}[x].
\ees
\end{definition}

\if{
Given a distribution $\kappa\in\calp_1(\tilde{\Gamma})$, we consider the cost functional $\tilde{J}^{\kappa}:=J[\mt^\kappa, \tilde{P}^\kappa]$, 
where $\mt^\kappa_t=\tilde{e}_t\sharp\kappa$, for $t\in[0,T]$ 
and where $\tilde{P}^\kappa(t)$ is constructed as an auxiliary function of the
distribution $\hat{e}_t\sharp\kappa$
at time
$t$
in Lemma  \ref{lemma:aux_ex_un}. The well-posedness of $\tilde{J}^\kappa$ is established in Remark \ref{rem:def-P-kappa}.
Once $\tilde{P}^\kappa$ has been defined, we can consider the set of optimal trajectories and associated adjoint states $\tilde{\Gamma}^\kappa[x_0]$ defined by
\begin{equation*}
\tilde{\Gamma}^\kappa[x_0]
= \Big\{ (\bar\gamma,p) \in\tilde{\Gamma}[x_0]\,:\, \bar\gamma\in \Gamma[\mt^\kappa, \tilde{P}^\kappa,x_0]
\text{ and }
p \text{ costate associated with }\bar{\gamma} \Big\}.
\end{equation*}
The precise meaning of ``associated costate" will be given in Definition \ref{def:costate}.

\begin{definition}
  A measure $ \kappa\in \mathcal{P}_{m_0} (\tilde{\Gamma})$
  is an {\it auxiliary MFGC equilibrium} if
\bes
\supp(\kappa)\subset \bigcup_{x\in\supp(m_0)}\tilde{\Gamma}^{\kappa}[x].
\ees
\end{definition}}
}\fi

We now establish the relationship between the notion of Lagrangian and auxiliary MFGC equilibria. Given $\kappa\in\calp(\tilde{\Gamma}_B)$, let $V^\kappa \colon \tilde{\Gamma}\to \controlset$ be defined by
\bes
V^\kappa(\gamma, p)=v\left[\gamma, \tilde{P}^\kappa+b(\gamma)^\top  p\right],
\ees
where the r.h.s.\@ is the Nemytskii operator associated with the auxiliary mapping introduced in Lemma \ref{lemma:aux_ex_un}.  
Let $\pi_1 \colon \tilde{\Gamma}\to\Gamma$ be such that $\pi_1(\gamma, p)=\gamma$. Then, we define $\eta[\kappa]=\left(\pi_1, V^\kappa\right)\sharp \kappa\in\calp_1(\Gamma\times\controlset)$.

\begin{lemma}
Let $\kappa\in\calp(\tilde{\Gamma})$ be an auxiliary MFGC equilibrium. Then, $\eta[\kappa]\in\calp_1(\Gamma\times\controlset)$ is a Lagrangian MFGC equilibrium.
\end{lemma}

\begin{proof} 
For the sake of simplicity we note $\eta$ instead of $\eta[\kappa]$. The main point is to prove that $\tilde{P}^\kappa= P^\eta$, where $P^\eta$ was introduced in \eqref{eq:def-m-P-eta-Lagrangian}. By the definition of $\eta$, it is supported on regular curves, thus
\begin{equation*}
\ba{lll}
P^\eta(t) &=& \psi \Big( \int_{\Gamma\times\controlset} v(t) \dd\eta(\gamma,v) \Big)\\[6pt]
\;&=& \psi \Big( \int_{\tilde{\Gamma}} v\left[\gamma(t), \tilde{P}^\kappa(t)+b(\gamma(t))^\top  p(t)\right]\dd \kappa(\gamma,p) \Big)\\[6pt]
\;&=& \psi \Big( \int_{\cR^n\times\cR^n} v\left[x, \tilde{P}^\kappa(t)+b(x)^\top  q\right]\dd \mu^\kappa_t(x,q)\Big) = \tilde{P}^\kappa(t).
\ea 
\end{equation*}
The last equality follows from \eqref{eq:cond_auxil} and \eqref{eq:def-tilde-P-kappa}. It is clear that $\mt^\kappa=m^\eta$, then $J^\eta=\tilde{J}^\kappa$ ($J^\eta$ is defined in Section \ref{subsec:LagrangianMFGCequilibria}). Since $\kappa$ is an auxiliary MFGC equilibrium, any $(\bar\gamma, \pb)\in \supp(\kappa)$ defines an optimal pair $(\bar\gamma, V^\kappa(\bar\gamma, \pb))$ for $J^\eta$. We conclude that $\eta$ is a Lagrangian MFGC equilibrium.
\end{proof}

In Section \ref{subsec:ExistenceResults} we show the existence of auxiliary MFGC equilibria, applying  Kakutani's fixed point theorem. The next technical section provides some convergence results to prove that the assumptions of  Kakutani's theorem hold.

\subsection{Convergence properties}

\begin{lemma}   \label{lem:conv-mu-P}
Let $(\kappa^i)_{i\in \NN}$ be a sequence contained in  $\calp_1(\tilde{\Gamma})$ 
 such that ${\rm supp}(\kappa^i) \subset \tilde{\Gamma}_B$ for all $i \in \NN$. Assume that $(\kappa^i)_{i \in \NN}$ narrowly converges to $\kappa$. Then, 
\bes
\sup_{t\in[0,T]} d_1(\mt^{\kappa^i}_t,\mt^\kappa_t)\rar 0\;\;\;\;\;\;\text{and}\;\;\;\;\;\sup_{t\in[0,T]}d_1(\mu^{\kappa^i}_t,\mu^\kappa_t)\rar 0.
\ees
\end{lemma}

\begin{proof}
We start proving that for any $\bar\kappa, \hat\kappa\in \calp_1(\tilde{\Gamma}_B)$ we have
\be \label{eq:continuity_marg}
\sup_{t\in[0,T]} d_1(m^{\bar\kappa}_t,m^{\hat\kappa}_t)\leq  d_1(\bar\kappa, \hat\kappa), \;\;\;\;\;{\rm and }\;\;\;\;\;\sup_{t\in[0,T]}d_1(\mu^{\bar\kappa}_t,\mu^{\hat\kappa}_t)\leq  d_1(\bar\kappa, \hat\kappa).
\ee
We show the result for $\mu^{\bar\kappa}_t$ and $\mu^{\hat\kappa}_t$, and then the result for $m^{\bar\kappa}_t$ and $m^{\hat\kappa}_t$ is straightforward. 
By the Kantorovich-Rubinstein formula, for any $t\in [0,T]$ we have
\begin{align*}
d_1(\mu^{\bar\kappa}_t,\mu^{\hat\kappa}_t)=&\ \sup_{\varphi \in {\rm Lip}_1(\cR^n\times\cR^n)}\int_{\cR^n\times\cR^n}\varphi(x,q)\dd(\mu^{\bar\kappa}_t-\mu^{\hat\kappa}_t)(x,q)\\
\;=& \ \sup_{\varphi\in {\rm Lip}_1(\cR^n\times\cR^n)}\int_{\tilde{\Gamma}}\varphi(\gamma(t),p(t))\dd(\bar\kappa-\hat\kappa)(\gamma, p) \ \leq \  \ d_1(\bar\kappa, \hat\kappa).
\end{align*}
In the last inequality we use the fact that given $\varphi\in {\rm Lip}_1(\cR^n\times\cR^n)$, the mapping $(\gamma, p)\mapsto \varphi(\gamma(t),p(t))$ belongs to ${\rm Lip}_1(\tilde{\Gamma})$, for all $t\in [0,T]$. 

Since $(\kappa^i)_{i \in \NN}\subset \calp_1(\tilde{\Gamma}_B)$ narrowly converges to $\kappa$, by \cite[Proposition 7.1.5]{ambrosio2008gradient}, 
we obtain $d_1(\kappa^i, \kappa)\to 0$. The conclusion follows with \eqref{eq:continuity_marg}. 
\end{proof}


\if{
\begin{lemma}   \label{lem:m-mu-hol}
Let $\kappa \in \calp_{m_0}(\tilde{\Gamma})$ with $\supp(\kappa)\subset \tilde{\Gamma}_B$. Then $\mt^\kappa_t$ and $\mu^\kappa_t$ are $\frac{1}{2}$-H\"older continuous w.r.t.\@ $t \in [0,T]$. 
\end{lemma}

\begin{proof} Recalling that $B=\overline{B}(0,M_1)\times \overline{B}(0,M_2)\subset \cR^n\times\cR^n$, since $\supp(\kappa)\subset \tilde{\Gamma}_B$, we obtain $\supp(\mu^\kappa_t)\subset B$ and $\supp(\mt^\kappa_t)\subset \bar{B}(M_1)$ for all $t\in [0,T]$. 

For all $s,t\in [0,T]$ we have
\begin{align*}
d_1(\mu^\kappa_t, \mu^\kappa_s)=&\ \sup_{\varphi\in {\rm Lip}_1(\cR^n\times\cR^n)}  \int_B\varphi(x,q)(\dd \mu^\kappa_t-\dd \mu^\kappa_s)(x,q)\\[6pt]
\;=&\ \sup_{\varphi\in {\rm Lip}_1(\cR^n\times\cR^n)}\int_\Gamma[\varphi(\gamma(t), p(t))-\varphi(\gamma(s), p(s))]\dd \kappa(\gamma, p)\\[6pt]
\;\leq & \ \int_\Gamma\max\{|\gamma(t)-\gamma(s)|,| p(t)- p(s)|\} \dd \kappa(\gamma, p)\\[6pt]
\;\leq& \ T^\frac{1}{2}\max\{M_3,M_4\}|t-s|^\frac{1}{2}.
\end{align*}
The last inequality holds by the assumption $\supp(\kappa)\subset \tilde{\Gamma}_B$.
\end{proof}

\begin{remark}    \label{rem:def-P-kappa}
Given $\kappa\in \calp(\tilde{\Gamma}_B)$, by  Lemma \ref{lem:conv-mu-P}, we obtain $\mt^\kappa\in C([0,T];\calp_1(\cR^n))$. 
Setting 
$
\tilde{P}^\kappa(t)=\mathbf{P}[\mu^\kappa_t]
$, 
where $\mathbf{P}$ is defined in Lemma \ref{lemma:aux_ex_un}, by Lemma \ref{lemma:aux_ex_un} and {\color{blue}Lemma \ref{lem:conv-mu-P}}, we obtain $\tilde{P}^\kappa\in C(0,T;\cR^m)$. Therefore, the definition of $\tilde{J}^\kappa$ in 
section \ref{auxMFGCequ} 
makes sense. 
\end{remark}
}\fi

\begin{lemma}
Let $(\kappa^i)_{i\in\NN}\subset \calp_{m_0}(\tilde{\Gamma})$, $\kappa\in \calp_{m_0}(\tilde{\Gamma})$ be such that ${\rm supp}(\kappa^i)\subset \tilde{\Gamma}_B$ for all $i$ and $\supp(\kappa)\subset \tilde{\Gamma}_B$. Assume that $\kappa^i$ narrowly converges to $\bar\kappa$. Let $(x_i)_{i\in\NN}\subset K_0$ be a sequence such that $x_i\to \xb$ and let 
$(\gamma_i,p_i)_{i\in\NN}\subset \tilde{\Gamma}^{\kappa^i}[x_i]$ 
(defined in section \ref{auxMFGCequ})
be a sequence such that $(\gamma_i, p_i)\to (\bar\gamma, \pb)$ uniformly on $[0,T]$. Then $(\bar\gamma, \pb)\in \tilde{\Gamma}^{\bar\kappa}[\xb]$. 
\label{lemma:closed-graph}
\end{lemma}

\begin{proof}
We have to prove that there exists $\vb\in L^2(0,T;\cR^m)$ such that  $(\bar\gamma,\vb)\in \calk[\xb]$ and
$$
\tilde{J}^{\bar\kappa}(\bar\gamma, \vb)\leq \tilde{J}^{\bar\kappa}(\gamma, v) \;\;\;\;\forall (\gamma, v)\in\calk[\xb].
$$ 
In addition, we have to prove that $\pb$ is the costate associated with $(\bar\gamma, \vb)$, in the sense of Definition \ref{def:costate}.

Since  $(\gamma_i,p_i)_{i\in\NN}\subset \tilde{\Gamma}^{\kappa^i}[x_i]$,  there exists for all $i\in\NN$ a control $v_i\in L^2(0,T;\cR^m)$ such that $(\gamma_i, v_i)\in\calk[x_i]$ and $(\gamma_i, v_i)$ is optimal for $\tilde{J}^{\kappa^i}$. By Proposition \ref{prop:bound_control_pb_gamma}, since $(x_i)_{i\in\NN}\subset K_0$, we have $\|\gamma_i(t)\|_{\infty}\leq M_1$ and $\|v_i\|_2\leq C$, for all $i\in\NN$. Therefore, there exists $\vb\in L^2(0,T;\cR^m)$ such that, up to a subsequence,  $v_i\rightharpoonup\vb$.  By Lemma \ref{lem:bound-deriv-gamma-p}, the sequence $\left(\gamma_i\right)_{i\in\NN}$ is a bounded sequence in $H^1(0,T;\cR^n)$, since $\gamma_i\to\bar\gamma$ in $C(0,T;\cR^n)$, it follows that $\bar\gamma\in H^1$ and $\dot{\gamma}_i\rightharpoonup\dot{\bar\gamma}$ in $L^2(0,T;\cR^n)$. In addition, by \hypReg{}-\hypRegA{}, \hypBound{}-\hypBoundA{},  the uniform convergence of $\gamma_i$ to $\bar\gamma$ and the weak convergence  of $v_i$ to $\vb$ we obtain
$$
a(\gamma_i)+b(\gamma_i)v_i\rightharpoonup a(\bar\gamma)+b(\bar\gamma)\vb, \;\;\;\;\text{in }L^2(0,T;\cR^n),
$$
which implies that 
$
\dot{\bar\gamma}(t)=a(\bar\gamma(t))+b(\bar\gamma(t))\vb(t)$, for a.e.\@ $t\in (0,T).$
It is clear that $\bar\gamma(0)=\xb$.

Furthermore, for all $i\in\NN$ there exists $(\lambda^i_1, \lambda_2^i,\nu_i)\in\cR^{n_{g_1}}\times\cR^{n_{g_2}}\times L^\infty(0,T;\cR^{n_c})$ such that \eqref{eq:opti_cond_oc1}, \eqref{eq:opti_cond_oc2} and \eqref{eq:opti_cond_oc3} hold for $(\gamma_i, v_i, p_i)$ and $\lambda_0^i=1$. By the proof of Proposition \ref{prop_bound_p}, we obtain that $(\lambda_1^i, \lambda_2^i)_{i\in\NN}$ is bounded, then there exists a subsequence, still denoted $(\lambda_1^i, \lambda_2^i)_{i\in\NN}$, that converges to $(\bar\lambda_1, \bar\lambda_2)$. 

By Lemma \ref{lemma:aux_arg_min_ham} and \eqref{eq:opti_cond_oc2}, we deduce
$$
v_i(t)=v\left[\gamma_i(t), \tilde{P}^{\kappa^i}(t)+b(\gamma_i(t))^\top p_i(t) \right], \;\;\nu_i(t)=\nu\left[\gamma_i(t), \tilde{P}^{\kappa^i}(t)+b(\gamma_i(t))^\top p_i(t) \right].
$$
By our assumptions, Lemma \ref{lemma:aux_ex_un} and Lemma \ref{lem:conv-mu-P}, the sequences $(\gamma_i)$, $(p_i)$ and $(\tilde{P}^{\kappa^i})$ are bounded and they converge to $\bar\gamma$, $\pb$ and $\tilde{P}^{\bar\kappa}$, uniformly over $[0,T]$. By Lemma \ref{lemma:aux_arg_min_ham}, the mappings $v[\cdot, \cdot]$ and $\nu[\cdot, \cdot]$ are Lipschitz continuous over bounded sets, then
\be   \label{eq:unif-conv-vi-vbar}
v_i(t)\to \vb(t)=v\left[\bar\gamma(t), \tilde{P}^{\bar\kappa}(t)+b(\bar\gamma(t))^\top\pb(t)\right], 
\ee
and
\bes
\nu_i(t)\to \bar\nu(t),\;\text{where }\; \bar\nu(t)=\nu\left[\bar\gamma(t), \tilde{P}^{\bar\kappa}(t)+b(\bar\gamma(t))^\top\pb(t)\right],
\ees
uniformly over $[0,T]$. In addition by Lemma \ref{lem:conv-mu-P}, $\sup_{t\in[0,T]}d_1(\mt^{\kappa^i}_t, \mt^{\bar\kappa}_t)\to 0$. Therefore by  \hypReg{} and \hypBound{}, we can pass to the limit in \eqref{eq:opti_cond_oc1}. By similar arguments we can pass to the limit in \eqref{eq:opti_cond_oc2} and \eqref{eq:opti_cond_oc3}. Finally we can conclude that $(\pb, 1, \bar\lambda_1, \bar\lambda_2, \bar\nu)$ satisfies the adjoint equation, the stationary condition and the complementarity condition for $(\bar\gamma, \vb)$.

Now, we prove the optimality of $(\bar\gamma, \vb)$ for $\tilde{J}^{\bar\kappa}$. First we show that
\be \label{eq:ineq-gamma-gamma-i}
\tilde{J}^{\bar\kappa}(\bar\gamma,\vb)= \lim_{i\to\infty}\tilde{J}^{\kappa^i}(\gamma_i, v_i).
\ee
By  the uniform convergence of the sequence $(\gamma_i)$,  Lemma \ref{lem:conv-mu-P} and \hypReg{} we have
\bes
\int_0^T f\left(\gamma_i(t), \mt^{\kappa^i}_t\right)\dd t\rar \int_0^Tf\left(\bar\gamma(t), \mt^{\bar\kappa}_t\right)\dd t\;\;\text{and}\;\;g_0\left(\gamma_i(T), \mt^{\kappa^i}_T\right)\rar g_0\left(\bar\gamma(T), \mt^{\bar\kappa}_T\right).
\ees
Skipping the time arguments, we have 
\bes
\int_0^T\left[\la \tilde{P}^{\kappa^i},v_i\ra-\la \tilde{P}^{\bar\kappa}, \vb\ra\right]\dd t=\int_0^T\left[\la \tilde{P}^{\kappa^i}- \tilde{P}^{\bar\kappa},v_i\ra+\la  \tilde{P}^{\bar\kappa}, v_i- \vb\ra\right]\dd t.
\ees
By Lemma \ref{lemma:aux_ex_un}, Lemma \ref{lem:conv-mu-P}, the uniform convergence in \eqref{eq:unif-conv-vi-vbar} and the boundedness of the sequences $(\tilde{P}^{\kappa^i})$ and $(v_i)$ we conclude that 
$
\int_0^T\la \tilde{P}^{\kappa^i}, v_i\ra\dd t\rar \int_0^T \la \tilde{P}^{\bar\kappa}, \vb\ra\dd t.
$
By \hypReg{}-\hypRegL{} and the uniform convergence of $(\gamma_i)$ and $(v_i)$ to $\bar\gamma$ and $\vb$, respectively, we deduce that
$
\int_0^T L(\gamma_i, v_i)\dd t\to \int_0^T L(\bar\gamma,\vb)\dd t.
$
Combining the above estimates, \eqref{eq:ineq-gamma-gamma-i} follows.

Now, let $(\hat\gamma, \vh)\in \calk[\xb]$ be an optimal solution for $\tilde{J}^{\bar\kappa}$ and initial condition $\xb$. By \hypQualif{}\hypQualifEq{}-\hypQualifIneq{}, Robinson's constraint qualification (see \cite[(2.163)]{BS13}) holds at $\vh$. By \cite[Theorem 2.87]{BS13} and \hypReg{} we conclude that there exists a sequence $(\vh_i)_{i\in\NN}\subset  L^\infty(0,T;\cR^m)$ such that $\|\vh_i-\vh\|_\infty\to 0$, and the sequence $(\hat\gamma_i)_{i\in\NN}$, given by 
$$
\left\{
\ba{rl}
\dot{\hat\gamma}_i(t)=& \! \! a(\hat\gamma_i(t))+b(\hat\gamma_i(t))\vh_i(t), \;\;\; \text{for a.e. }t\in[0,T]\\[6pt]
\hat\gamma_i(0)=& \! \! x_i
\ea
\right.
$$
is such that $(\hat\gamma_i,\vh_i)\in \calk[x_i]$. In addition, by our assumptions and Gr\"onwall's Lemma we deduce that $(\hat\gamma_i)_{i\in\NN}$ is uniformly bounded in $L^\infty(0,T;\cR^n)$ and $\|\hat\gamma_i-\hat\gamma\|_\infty\to 0$. 

By the optimality of  $(\gamma_i, p_i)\in \tilde{\Gamma}^{\kappa^i}[x_i]$, we have
\be
\tilde{J}^{\kappa^i}(\gamma_i, v_i)\leq \tilde{J}^{\kappa^i}(\hat\gamma_i,\vh_i) \;\;\;\forall i\in\NN.    \label{eq:ineq-hat-gamma-i}
\ee
Since $\|\vh_i-\vh\|_\infty\to 0$ and $\|\hat\gamma_i-\hat\gamma\|_\infty\to 0$, arguing as above we obtain, 
\be
\lim_{i\to\infty}\tilde{J}^{\kappa^i}(\hat\gamma_i,\vh_i)=\tilde{J}^{\bar\kappa}(\hat\gamma,\vh). \label{eq:lim-hat-gamma-i}
\ee
By \eqref{eq:ineq-gamma-gamma-i}, \eqref{eq:ineq-hat-gamma-i}  and \eqref{eq:lim-hat-gamma-i}, we deduce
\bes
\tilde{J}^{\bar\kappa}(\bar\gamma, \vb)= \lim_{i\to\infty}\tilde{J}^{\kappa^i}(\gamma_i,v_i)\leq \lim_{i\to\infty}\tilde{J}^{\kappa^i}(\hat\gamma_i,\vh_i)=\tilde{J}^{\bar\kappa}(\hat\gamma,\vh).  
\ees
Then, $(\bar\gamma,\vb)$ is optimal, which finally proves that $(\bar\gamma, \pb)\in\tilde{\Gamma}^{\bar\kappa}[\xb]$.
\end{proof}

\subsection{Existence results}  \label{subsec:ExistenceResults} 

In this section, we characterize auxiliary MFGC equilibria as fixed points of a set-valued map. Applying Kakutani's fixed point theorem, we prove the existence of such equilibria. 

By \cite[Theorem 5.3.1]{ambrosio2008gradient} (Disintegration Theorem), 
for any $\kappa\in \calp_{m_0}(\tilde{\Gamma})$, there exists a $m_0$-a.e. uniquely determined Borel measurable family $\{\kappa_x\}_{x\in\RR^n}\subset \calp(\tilde\Gamma)$ such that
\bes
{\rm supp}(\kappa_x)\subset \tilde\Gamma[x],\;\; m_0{\rm -a.e.}\; x\in\RR^n,
\ees
and for any Borel mapping $\varphi \colon \tilde\Gamma\rar [0,+\infty]$,
\bes
\int_{\tilde{\Gamma}}\varphi(\gamma, p)\dd \kappa(\gamma, p)=\int_{\RR^n}\Big( \int_{\tilde{\Gamma}[x]}\varphi(\gamma, p)\dd \kappa_x(\gamma, p)\Big)\dd m_0(x).
\ees
Following the lines of \cite{Cannarsa-Capuani}, we define the set-valued map $E:\calp_{m_0}(\tilde{\Gamma})\rightrightarrows \calp_{m_0}(\tilde{\Gamma})$ as
\bes
E(\kappa)=\{\hat\kappa\in \calp_{m_0}(\tilde{\Gamma}): \supp(\hat\kappa_x)\subset \tilde\Gamma^\kappa[x], \;m_0{\rm -a.e. } \;\;x\in\cR^n \}.
\ees
It follows that $\kappa$ is an auxiliary MFGC equilibrium if and only if $\kappa\in E(\kappa)$.

\begin{theorem}  \label{Th:existence-extended-eq}
There exists at least one auxiliary MFGC equilibrium.
\end{theorem}

\begin{proof}
Arguing as in \cite[Lemma 3.5]{Cannarsa-Capuani}, for any $\kappa\in\calp_{m_0}(\tilde{\Gamma})$ the set $E(\kappa)$ is a nonempty convex set.
 By Proposition \ref{prop:bound_control_pb_gamma}, Proposition \ref{prop_bound_p}  and Lemma \ref{lem:bound-deriv-gamma-p} we have
\bes
E(\kappa)\subset \calp_{m_0}(\tilde\Gamma_B), \;\;\;\;\forall \kappa\in\calp_{m_0}(\tilde{\Gamma}),
\ees
where $\tilde\Gamma_B$ was introduced in \eqref{eq:Gamma_B}. By Lemma \ref{lemma:closed-graph}, and \cite[Lemma 3.6]{Cannarsa-Capuani}, we conclude that  the map $E \colon \calp_{m_0}(\tilde{\Gamma})\rightrightarrows \calp_{m_0}(\tilde{\Gamma})$ has closed graph.

Finally,  since the set $\tilde\Gamma_B$ is a compact subset of $\tilde{\Gamma}$, we obtain that $\calp_{m_0}(\tilde\Gamma_B)$ is a nonempty compact convex set. Then, we can apply Kakutani's fixed point theorem, to deduce that there exists $\hat\kappa\in \calp_{m_0}(\tilde\Gamma_B)$ such that $\hat\kappa\in E(\hat\kappa)$.
\end{proof}

\begin{remark} \label{rem:impossibility}
Let us comment on the impossibility to employ a similar fixed point approach directly based on the notion of Lagrangian equilibria (Definition \ref{def:lagrangian}). Consider a probability distribution $\eta$ of state-control trajectories. From the definition of $P^\eta$, there is no regularity property (with respect to time) to expect, since the controls in problem \eqref{eq:control_prob} are taken in $L^2(0,T;\cR^m)$.
Consequently, it is not possible to use relation \eqref{eq:reg_control} to derive any regularity property for the optimal controls with respect to the criterion $J[m^\eta,P^\eta]$ and thus it does not seem possible to construct an appropriate compact set of probability distributions of state-control trajectories, on which some fixed point relation could be defined.
\end{remark}

\section{Uniqueness}      \label{sec:uniqueness}

As usual in the MFG theory, by adding some monotonicity assumptions we can obtain uniqueness results.

\begin{definition}
A function $\varphi \colon \cR^n\times \calp(\cR^n)\to \cR$ is monotone if
$$
\int_{\cR^n}\left(\varphi(x,m_1)-\varphi(x,m_2)\right)\dd \left(m_1-m_2 \right)(x)\geq 0, \;\;\;\;\forall m_1,\;m_2\in\calp(\cR^n).
$$
It is strictly monotone if it is monotone and 
$$
\int_{\cR^n}\left(\varphi(x,m_1)-\varphi(x,m_2)\right)\dd \left(m_1-m_2 \right)(x) = 0, 
$$
if and only if $\varphi(x,m_1)=\varphi(x,m_2)$ for all $x\in\cR^n$. 
\end{definition}

An example of strictly monotone function can be found in \cite{Cannarsa-Capuani}.

\begin{theorem}
Assume that $f$ and $g_0$ are strictly monotone and  $\psi$ is also strictly monotone (i.e.\@ for all $x$ and $y \in \cR^m$ with $x \neq y$, $\langle \psi(y)-\psi(x), y-x \rangle > 0$).  
Let $\eta_1, \eta_2\in\calp_{m_0}(\Gamma\times\controlset)$ be Lagrangian MFGC equilibria for $m_0$, then $P^{\eta_1}=P^{\eta_2}$ and $J^{\eta_1}=J^{\eta_2}$.
\end{theorem}

\begin{proof}
Let us define 
$
u^i(x)= \inf_{(\gamma, v)\in \calk[x]}J^{\eta_i}(\gamma, v)$, $i=1,2.
$
Let $(\gamma,v)\in \supp(\eta_1)$, then 
\begin{align*}
u^1(\gamma(0))= & \int_0^T\left(L(\gamma(t), v(t)) + \langle P^{\eta_1}(t), v(t)\rangle + f(\gamma(t), m^{\eta_1}_t) \right)\dd t+g_0(\gamma(T), m^{\eta_1}_T), \\
u^2(\gamma(0))\leq & \int_0^T\left(L(\gamma(t), v(t)) + \langle P^{\eta_2}(t), v(t)\rangle + f(\gamma(t), m^{\eta_2}_t) \right)\dd t+g_0(\gamma(T), m^{\eta_2}_T).
\end{align*}
Integrating w.r.t. $\eta_1$ we obtain
\begin{align*}
& \int_{\Gamma\times\controlset} (u^1(\gamma(0))-u^2(\gamma(0)))\dd \eta_1(\gamma, v) + \int_{\Gamma\times\controlset}\int_0^T \langle P^{\eta_2}(t)-P^{\eta_1}(t), v(t)\rangle \dd t\dd \eta_1(\gamma, v) \\[6pt]
& \qquad \geq 
\int_{\Gamma\times\controlset}\left( g_0(\gamma(T), m^{\eta_1}_T)-g_0(\gamma(T), m^{\eta_2}_T)\right)\dd \eta_1(\gamma, v)\\[6pt]
& \qquad \qquad +\int_{\Gamma\times\controlset}\int_0^T \left( f(\gamma(t), m^{\eta_1}_t)-f(\gamma(t), m^{\eta_2}_t)\right)\dd t\dd \eta_1(\gamma, v). 
\end{align*}
By the definition of $m^{\eta_1}$ we obtain
\begin{align*}
& \int_{\cR^n} (u^1(x)-u^2(x))\dd m_0(x)+\int_{\Gamma\times\controlset}\int_0^T \langle P^{\eta_2}(t)-P^{\eta_1}(t), v(t)\rangle \dd t\dd \eta_1(\gamma, v)\\[6pt]
& \qquad \geq\int_{\cR^n}\left( g_0(x, m^{\eta_1}_T)-g_0(x, m^{\eta_2}_T)\right)\dd  m^{\eta_1}_T(x)\\[6pt]
& \qquad \qquad +\int_0^T\int_{\cR^n} \left( f(x, m^{\eta_1}_t)-f(x, m^{\eta_2}_t) \right)\dd m^{\eta_1}_t(x)\dd t.
\end{align*}
Arguing in a similar way for $\eta_2$, we deduce
\begin{align}
& \int_0^T\int_{\Gamma\times\controlset} \langle P^{\eta_2}(t)-P^{\eta_1}(t), v(t)\rangle \dd (\eta_1-\eta_2)(\gamma, v)\dd t \notag \\[6pt]
& \qquad \geq \int_{\cR^n}\left( g_0(x, m^{\eta_1}_T)-g_0(x, m^{\eta_2}_T)\right)\dd \left(m^{\eta_1}_T-m^{\eta_2}_T\right)(x) \notag \\[6pt]
& \qquad \qquad +\int_0^T\int_{\cR^n} \left( f(x, m^{\eta_1}_t)-f(x, m^{\eta_2}_t) \right)\dd \left(m^{\eta_1}_t-m^{\eta_2}_t\right)(x) \dd t. \label{ineq:f-g-monotone}
\end{align}
By the definition of $P^{\eta_i}$ we deduce
\begin{align*}
& \int_0^T\int_{\Gamma\times\controlset} \langle P^{\eta_2}(t)-P^{\eta_1}(t), v(t)\rangle\dd (\eta_1-\eta_2)(\gamma, v)  \dd t\\[6pt]
& \quad = \int_0^T\Big\langle \psi \left({\textstyle \int} v\dd\eta_2(\gamma, v)\right)- \psi \left( {\textstyle \int}  v\dd\eta_1(\gamma, v)\right), {\textstyle \int}  v\dd\eta_1(\gamma, v)- {\textstyle \int}  v\dd\eta_2(\gamma, v)  \Big\rangle \dd t
\end{align*}
and the r.h.s. is non-positive,  by Assumption \hypConv{}-\hypConvPsi{}. In addition, since $f$ and $g_0$ are monotone, we deduce that the three terms in \eqref{ineq:f-g-monotone} vanish. Since $f$ and $g_0$ are strictly monotone we obtain for all $x\in\cR^n$ and a.e. $t\in (0,T)$,
\bes
f(x, m^{\eta_1}_t)=f(x, m^{\eta_2}_t) \quad \text{and }\quad g_0(x, m^{\eta_1}_T)=g_0(x, m^{\eta_2}_T).
\ees 
By the strict monotony of $\psi$ we have
\bes 
\int v\dd\eta_1(\gamma, v) = \int v\dd\eta_2(\gamma, v), \;\;\;\text{a.e.}\;t\in (0,T),
\ees 
which in particular implies $P^{\eta_1}=P^{\eta_2}$. The result follows. 
\end{proof}

\begin{remark} As noted in \cite{Cannarsa-Capuani}, if we assume that $\psi$ is strictly monotone, $g_0$ is monotone and $f$ satisfies
$$
\int_{\cR^n}\left(f(x,m_1)-f(x,m_2)\right)\dd \left(m_1-m_2 \right)(x)\leq 0\;\Rightarrow m_1=m_2,
$$
then, following the ideas of the above proof, we obtain $P^{\eta_1}=P^{\eta_2}$ and $m^{\eta_1}=m^{\eta_2}$. 
\end{remark}

\section{Conclusion}

We have proved the existence of a Lagrangian equilibrium for an MFG of controls with final state and mixed state-control constraints, and a class of nonlinear dynamics. Using auxiliary mappings and a priori estimates on optimal state-costate trajectories, we have reformulated the problem as a fixed point problem on a compact set of probability measures on state-costate trajectories. As explained in Remark \ref{rem:impossibility}, this reformulation was necessary, in the absence of smoothing properties of the price interaction.



A future direction of research may concern the characterization of the equilibrium with a system of coupled partial differential equations (HJB equation and continuity equation), as it is done for example in \cite{cannarsa2018mean}. In this reference, a feedback control is constructed thanks to the differentiability of the value function, itself obtained with the strict convexity of the Hamiltonian. This last property is however lost (in general) in the presence of mixed state-control constraints.
Another difficulty would arise from the treatment of final-state constraints. The recent work \cite{bokanowski2021relationship} may contain useful tools in that direction; this article
deals with optimal control problems with final-state constraints: it provides a characterization of the value function as well as sensitivity relation.


In some future work, one could also address the extension of our aggregative MFG model to the case of pure state constraints, as those considered in \cite{Cannarsa-Capuani}. As we already pointed out, our analysis relies in a crucial way on some a priori estimates on the costate, whose evolution is not impacted by the price variable. Proving the regularity of the costate, in the presence of pure state constraints and a merely measurable price function, seems however to be a great challenge.




\appendix

\section{Proof of optimality conditions} \label{section:proof_oc}

We provide in this section a proof of the optimality conditions stated in Proposition \ref{prop:opti_cond}.
An important difficulty is the fact that optimal controls are not a priori known to be bounded (we are not able to prove the boundedness of optimal controls without having the optimality conditions at hand).
It is therefore not possible to formulate the optimal control problem as an abstract problem satisfying a qualification condition in $L^\infty$ and to derive easily optimality conditions, as it is done in \cite{bonnans14} for example.
It turns out that the optimal control problem can be naturally formulated as an optimal control problem for which the dynamic constraint takes the form of a differential inclusion. This enables us to use the associated optimality conditions, referred to as \emph{extended Euler-Lagrange} conditions in the literature. More precisely, our analysis is based on \cite[Theorem 7.5.1]{Vin00}, which covers the case of unbounded controls and requires few regularity assumptions.

We first introduce two definitions of cones, used for the expression of the optimality conditions for problems with differential inclusions.
Given a closed subset $K$ of $\cR^{\ell}$ and 
 $x \in K$, we call proximal normal cone of $K$ at $x$ the set $N_K^P(x)$ defined by
\begin{equation*}
N_K^P(x)= \big\{
p \in \cR^{\ell} : \exists C> 0, \, \forall y \in K, \, \langle p, y-x \rangle \leq  C |y-x|^2
\big\}.
\end{equation*}
That is, $p \in N_K^P(x)$ if and only if, for some $C>0$,
\be \label{PNC}
x\in \argmin\;\{ \langle -p, y \rangle + C |y-x|^2\;: \, y\in K\}. 
\ee
The limiting normal cone $N_K(x)$ is defined by
\begin{equation*}
N_K(x)= \left\{ p \in \cR^{\ell} \, \Big|\, \exists (x_k,p_k)_{k \in \mathbb{N}} \text{ such that: }
\begin{array}{l}
(x_k,p_k) \rightarrow (x,p), \text{ as $k\to \infty$} \\[0.3em]
x_k \in K, \ p_k \in N_K^P(x_k)\;\; \forall k \in \mathbb{N}
\end{array}
\right\}.
\end{equation*}

\begin{proof}[Proof of Proposition \ref{prop:opti_cond}]

\emph{Step 1:} reformulation of the optimal control problem.
Let us fix a solution $(\bar{\gamma},\bar{v}) \in H^1(0,T;\cR^n) \times L^2(0,T;\cR^m)$ to \eqref{eq:control_prob}.
In order to alleviate the notation, we first define
\begin{equation*}
\tilde{L}(t,x,v)= L(x,v)+ \langle P(t), v \rangle + f(x,m(t)),
\end{equation*}
for all $(x,v) \in \cR^{n+m}$ and for a.e.\@ $t \in (0,T)$.

\if{
{\color{blue}We work with an augmented state variable $y= (y^{(1)},y^{(2)},y^{(3)}) \in \cR^{n+m+1}$.
We consider a set-valued map $F \colon [0,T] \times \cR^{n+m+1} \rightrightarrows \cR^{n+m+1}$. Given $y=(x,v,z) \in \cR^{n+m+1}$, we set
$
F(t,y)= \big\{
\xi(t,y) :\, c(x,v) \leq 0, \, z \geq 0
\big\},
$
where
\begin{equation*}
\xi^{(1)}(t,y)= a(x) + b(x) v, \;\;
\xi^{(2)}(t,y)=  v,\;\;
\xi^{(3)}(t,y)=  \tilde{L}(t,x,v) + z.
\end{equation*}
}
}\fi

We work with an augmented state variable $y= (y^{(1)},y^{(2)},y^{(3)}) \in \cR^{n+m+1}$.
We consider a set-valued map $F \colon [0,T] \times \cR^{n+m+1} \rightrightarrows \cR^{n+m+1}$ defined as  
$
F(t,y)= \big\{
\xi(t,y^{(1)},v,z) :\,(v,z)\in\cR^{m}\times\cR,\, c(y^{(1)},v) \leq 0, \, z \geq 0
\big\},
$
where  for $(x,v,z)\in  \cR^{n+m+1}$ 
\begin{equation*}
\xi^{(1)}(t,x,v,z)= a(x) + b(x) v, \;\;
\xi^{(2)}(t,x,v,z)=  v,\;\;
\xi^{(3)}(t,x,v,z)=  \tilde{L}(t,x,v) + z.
\end{equation*}
The component $\xi^{(1)}$ coincides with the dynamics of the original state variable. The second component has a technical purpose, it allows in particular to prove easily that $F(t,y)$ is closed (which would be delicate otherwise, since the controls are not necessarily bounded). The third component allows to put the problem in Mayer form. 
The initial condition associated with the new state variable is defined by 
$
\bar{y}_0 = (x_0,0,0) \in \cR^{n+m+1}
$. Let $K \subseteq \cR^{2(n+m+1)}$ be given by
\begin{equation*}
K= \big\{ (y_i,y_f) \in \cR^{2(n+m+1)} : y_i= \bar{y}_0,\, \, g_1 \big( y_f^{(1)} \big)= 0, \, g_2 \big( y_f^{(1)} \big) \leq 0 \big\}.
\end{equation*}
We define $\Phi \colon \cR^{2(n+m+1)} \rightarrow \cR$ by 
$
\Phi(y_i,y_f)= g_0(y_f^{(1)}) + y_f^{(3)}
$. The optimal control problem \eqref{eq:control_prob} can finally be reformulated as follows:
\begin{equation} \label{pb:opti_inclusion}
\inf_{y \in H^1(0,T;\cR^{n+m+1})} \Phi(y(0),y(T)), \;\;
\text{subject to: }
\left\{
\begin{array}{l}
\dot{y}(t) \in F(t,y(t)), \ \text{for a.e. $t \in (0,T)$}, \\
(y(0),y(T)) \in K.
\end{array}
\right.
\end{equation}
More precisely, the trajectory $\bar{y}$, defined by
\begin{equation*}
\begin{cases}
\bar{y}^{(1)}(t)= \bar{\gamma}(t) \\
\bar{y}^{(2)}(t)= \int_0^t \bar{v}(s) \dd s \\
\bar{y}^{(3)}(t)= \int_0^t \tilde{L}(s,\bar{\gamma}(s),\bar{v}(s)) \dd s
\end{cases}
\end{equation*}
is a solution to \eqref{pb:opti_inclusion}.
Denoting $\bar{\xi}= \dot{\bar{y}}$, we note that $\bar{\xi}(t)= \xi(t,\bar{\gamma}(t),\bar{v}(t),0)$.

\emph{Step 2:}
verification of the technical conditions of \cite[Theorem 7.5.1]{Vin00}.
It is easily verified that for a.e.\@ $t \in (0,T)$, $F(t,y)$ is non-empty and convex, as a consequence of Assumptions \hypConv-\hypConvL, \hypConv-\hypConvC, and \hypFeas-\hypFeasLocal. It is also easily verified that $F$ is measurable and has a closed graph. It remains to show that there exist $\eta > 0$ and $k \in L^1(0,T)$ such that
\begin{equation} \label{eq:technical_inclusion}
F(t,\tilde{y}) \cap \big( \dot{\bar{y}}(t) + \eta k(t) \bar{B}(1) \big)
\subseteq F(t,y) + k(t) |\tilde{y}-y| \bar{B}(1),
\end{equation}
for a.e.\@ $t \in (0,T)$ and for all $y$ and $\tilde{y}$ such that $|y-\bar{y}(t)| \leq \eta$ and $|\tilde{y}-\bar{y}(t)| \leq \eta$.
Let $t \in (0,T)$, let $y$ and $\tilde{y}$ be such that $|y-\bar{y}| \leq \delta/2$ and $|\tilde{y}-\bar{y}| \leq \delta /2$, where $\delta$ is given by Lemma \ref{lemma:reg_metric_mixed}, with $R= \| \bar{\gamma} \|_{L^\infty(0,T;\cR^n)}$.
Let $\bar{k}(t)= 1 + |\bar{v}(t)|^2$.
Let $\tilde{\xi} \in F(t,\tilde{y}) \cap \big( \dot{\bar{y}}(t) +  \bar{k}(t) \bar{B}(1) \big)$. Let $\tilde{v} \in \cR^m$ and $\tilde{z} \in \cR$ be such that $\tilde\xi = \xi(t,\tilde{y}^{(1)},\tilde{v},\tilde{z})$, $c(\tilde{y}^{(1)},\tilde{v}) \leq 0$ and $\tilde{z} \geq 0$.
Since $|\tilde{\xi} - \bar{\xi}(t) | \leq \bar{k}(t)$, we deduce that
\begin{equation*}
|\tilde{\xi}^{(2)}-\bar{\xi}^{(2)}(t)|
= | \tilde{v} - \bar{v}(t) | \leq \bar{k}(t).
\end{equation*}
Therefore
\begin{equation*}
|\tilde{v}| \leq |\tilde{v} - \bar{v}(t)| + |\bar{v}(t)|
\leq \bar{k}(t) + \frac{1}{2} + \frac{1}{2} |\bar{v}(t)|^2
\leq \frac{3}{2} \bar{k}(t).
\end{equation*}
We also have 
$
|y^{(1)}-\tilde{y}^{(1)}| \leq | y - \bar{y}(t) | + | \tilde{y}- \bar{y}(t) | \leq \delta
$. 
Thus by Lemma \ref{lemma:reg_metric_mixed}, there exists $v \in \cR^m$ such that $c(y^{(1)},v) \leq 0$ and $|v-\tilde{v}| \leq C |y^{(1)}-\tilde{y}^{(1)}|$ (note that all constants $C$ involved for the verification of \eqref{eq:technical_inclusion} are independent of $(t,\tilde{y},\tilde{v},y,v)$).
Let $\xi= \xi(t,y^{(1)},v,\tilde{z})$. We have $\xi \in F(t,y)$.
It remains to bound $|\xi- \tilde{\xi}|$. We first have
\begin{align*}
|\xi^{(1)}-\tilde{\xi}^{(1)}|
\leq \ & |a(y^{(1)})-a(\tilde{y}^{(1)}) |
+|b(y^{(1)})| \cdot |v - \tilde{v}| + |b(y^{(1)})-b(\tilde{y}^{(1)})| \cdot |\tilde{v}| \\
\leq \ & C \big( |y- \tilde{y}| + |v- \tilde{v}| \big) ( 1 + |\bar{v}(t)|^2 ) \\
\leq \ & C |y-\tilde{y} | \bar{k}(t),
\end{align*}
by \hypReg{}-\hypRegA{}.
The same estimate can be established for $|\xi^{(3)}-\tilde{\xi}^{(3)}|$ (with the help of Assumption \hypBound{}-\hypBoundL{}) and for $|\xi^{(2)}-\tilde{\xi}^{(2)}|$, thus
\begin{equation} \label{eq:opti_cond_opt_1}
| \xi - \tilde{\xi} |
\leq C |y-\tilde{y} | \bar{k}(t).
\end{equation}
The inclusion \eqref{eq:technical_inclusion} follows, taking $k(t)= C \bar{k}(t)$ and $\eta= \min \big( \delta/2, 1/C \big)$, where $C$ is the constant appearing in the right-hand side of \eqref{eq:opti_cond_opt_1}.

\emph{Step 3:} abstract optimality conditions and interpretation.
Applying \cite[Theorem 7.5.1]{Vin00}, we obtain the existence of $\bar{p} \in W^{1,1}(0,T;\cR^{n+m+1})$ and $\lambda_0 \geq 0$ such that:
\begin{enumerate}
\item[(i)] $(\bar{p},\lambda_0) \neq (0,0)$,
\item[(ii)] $-\dot{\bar{p}}(t) \in \text{conv} \big\{ q :\, (q,-\bar{p}(t)) \in N_{\text{Gr}(F(t,\cdot))}(\bar{y}(t),\bar{\xi}(t)) \big\}$,
\item[(iii)] $(-\bar{p}(0),\bar{p}(T)) \in \lambda_0 \nabla \Phi(\bar{y}(0),\bar{y}(T)) + N_K(\bar{y}(0),\bar{y}(T))$,
\end{enumerate}
where $\text{Gr}(F(t,\cdot))= \{ (y,\xi) : \, \xi \in F(t,y) \}$.
We let the reader verify that the condition (iii) (together with Assumptions \hypQualif{}-\hypQualifEq{} and \hypQualif{}-\hypQualifIneq{}) implies the existence of $\lambda_1 \in \cR^{n_{g_1}}$ and $\lambda_2 \in \cR^{n_{g_2}}$, $\lambda_2 \geq 0$, such that
\begin{equation*}
\begin{array}{rl}
\bar{p}^{(1)}(T)^\top= & \lambda_0 Dg_0(\bar{y}^{(1)}(T)) + \lambda_1^\top D g_1(\bar{y}^{(1)}(T)) + \lambda_2^\top D g_2(\bar{y}^{(1)}(T)), \\
\bar{p}^{(2)}(T)^\top= & 0, \\
\bar{p}^{(3)}(T)^\top= & \lambda_0,
\end{array}
\end{equation*}
and such that $\langle g_2(\bar{y}^{(1)}(T)),\lambda_2 \rangle= 0$.
For the interpretation of the adjoint equation (condition (ii)), we
need to examine the limiting normal cone of the graph of
$F(t,\cdot)$. Let $y \in \cR^{n+m+1}$, let $\xi \in F(t,y)$, and let
$(q,-p) \in N_{\text{Gr} (F(t,\cdot))}(y,\xi)$. Let $y_k \rightarrow
y$, $\xi_k= \xi(t,y_k^{(1)},v_k,z_k) \rightarrow \xi$, $\xi_k \in F(t,y_k)$,
$(q_k,p_k) \rightarrow (q,p)$ be such that $(q_k,-p_k) \in
N_{\text{Gr}(F(t,\cdot))}^P(y_k,\xi_k)$. By definition of the proximal
normal cone, 
  see \eqref{PNC}, $(y_k,\xi_k)$ is for some $C>0$ (depending on $k$) solution of the minimization
  problem
\bes
  \Min_{(y,\xi)\in \text{Gr}(F(t,\cdot))} \sum_{i=1}^3 \left(
   \langle -q^{(i)}_k, y^{(i)} \rangle 
+
\langle p^{(i)}_k, \xi^{(i)} \rangle  \right)
+
C  ( |  y-y_k|^2 + |  \xi-\xi_k|^2 ).
\ees
In view of the expression of the multimapping $F$, 
this holds if and only if, for some $(v_k,z_k)\in \cR^{m}\times\cR$,
$(y_k,v_k,z_k)\in \cR^{n+m+1}\times\cR^{m}\times\cR$ is solution of
\bes
\ba{lll}
  \Min\limits_{(y,v,z)} \sum_{i=1}^3 
   \langle -q^{(i)}_k, y^{(i)} \rangle 
+
\langle p^{(1)}_k,  a(y^{(1)}) + b(y^{(1)}) v  \rangle
\\ \quad \quad + \langle p^{(2)}_k, v \rangle  
+
\langle p^{(3)}_k, \tilde{L}(t,y^{(1)},v) + z \rangle  +

C  ( |  y-y_k|^2 + |  \xi-\xi_k|^2 ),
\\[6pt]
\quad \text{s.t. $c(y^{(1)},v) \leq 0$ and $z \geq 0$}.
\ea
\ees
Since this problem is qualified, we obtain the existence of $\nu_k \in \cR^{n_c}$, $\nu_k \geq 0$, such that the following stationarity and complementarity conditions hold:
\begin{itemize}
\item Stationarity with respect to $z$: $p^{(3)}_k \geq 0$.
\item Stationarity with respect to $v$: 
\begin{align}  \label{eq:stationarity-v}
(p_k^{(1)})^\top b(y_k^{(1)}) + (p_k^{(2)})^\top + p_k^{(3)} D_v \tilde{L}(y_k^{(1)},v_k) 
+ \nu_k^\top D_v c(y_k^{(1)},v_k)= 0.
\end{align}
\item Stationarity with respect to $y^{(1)}$:
\begin{equation*}
\ba{l}
-(q_k^{(1)})^\top + (p_k^{(1)})^\top
\Big(
Da(y_k^{(1)}) + \sum_{i=1}^m Db(y_k^{(1)}) v_{k,i}
\Big)\\[5pt]
\hspace{3cm} + p_k^{(3)} D_x \tilde{L}(y_k^{(1)},v_k)
+ \nu_k^\top D_x c(y_k^{(1)},v_k)
= 0.
\ea
\end{equation*}
\item Stationarity with respect to $y^{(2)}$: $q_k^{(2)}= 0$.
\item Stationarity with respect to $y^{(3)}$: $q_k^{(3)}= 0$.
\item Complementarity: $\langle c(y_k^{(1)},v_k), \nu_k \rangle = 0$.
\end{itemize}

The inward pointing condition, Assumption \hypQualif{}-\hypQualifInward{}, yields a uniform bound on $\nu_k$
(with respect to $k$). This allows to pass to the limit in the above
relations, using the continuity assumptions on $\tilde{L}$, $a$, $b$,
and $c$ (note that $\xi_k^{(2)} \rightarrow \xi^{(2)}$ implies that
$v_k \rightarrow v$).  We deduce that $\dot{\bar{p}}^{(2)}= 0$ and
$\dot{\bar{p}}^{(3)}= 0$, thus $\bar{p}^{(2)}(t)= 0$ and
$\bar{p}^{(3)}(t)= \lambda_0$.  
Since $\pb\in W^{1,1}(0,T;\cR^{n+m+1})$, it belongs to $L^\infty(0,T;\cR^{n+m+1})$. Passing to the limit in \eqref{eq:stationarity-v} we obtain 
\be\label{eq:app-lam-0}
  \lambda_0 D_v \tilde{L}(\yb^{(1)}(t),\vb(t)) +(\pb^{(1)}(t))^\top b(\yb^{(1)}(t))+ \bar\nu(t)^\top D_v c(\yb^{(1)}(t),\vb(t))= 0.
\ee 
If $\lambda_0\neq 0$, $\bar\nu(t)=\nu\left[\yb^{(1)}(t), \lambda_0 P(t)+(\pb^{(1)}(t))^\top b(\yb^{(1)}(t))\right]$, by  Lemma \ref{lemma:aux_arg_min_ham}.
Since $\nu[\cdot, \cdot]$ is Lipschitz continuous on bounded sets, we obtain $\bar\nu\in L^\infty(0,T;\cR^{n_c})$. Analogously, $\vb \in L^\infty(0,T;\cR^m)$, therefore we  deduce that $\pb\in W^{1,\infty}(0,T;\cR^{n+m+1})$.

If $\lambda_0=0$, we denote by $\bar\nu_I(t)$ the components of $\bar\nu(t)$ whose indices belong to the set $I(\yb^{(1)}(t), \vb(t))$. Skipping the time arguments, from \eqref{eq:app-lam-0} we deduce 
$$
\bar\nu_I = - \left[D_vc_{I}(\yb^{(1)},\vb)D_vc_{I}(\yb^{(1)},\vb)^\top \right]^{-1}D_vc_{I}(\yb^{(1)},\vb) b(\yb^{(1)})^\top\pb^{(1)}.
$$
The matrix $\left[D_vc_{I}(\yb^{(1)},\vb)D_vc_{I}(\yb^{(1)},\vb)^\top \right]$ is uniformly invertible by  \hypQualif{}-\hypQualifLI{}. Since $\yb^{(1)}, \pb^{(1)}\in L^\infty(0,T;\cR^{n})$, by \hypReg{}-\hypRegL{}, \hypBound{}-\hypBoundC{}  and \hypBound{}-\hypBoundA{} we deduce that $\bar\nu\in L^\infty(0,T;\cR^{n_c})$. In this case, we only have $\vb\in L^2(0,T;\cR^m)$, so we obtain $\pb\in W^{1,2}(0,T;\cR^{n+m+1})$.
\end{proof}

\section{Application to a gas storage problem}
\subsection{Setting}

Consider the case when the scalar state $\gamma(t)$
represents a scaled energy storage,
with value in $[0,1]$ and integrator dynamics
\bes
\dot \gamma(t) = v(t).
\ees
Therefore $a(x)=0$ and $b(x)=1$.
In addition we have limitations on the efficiency of
pumping depending on the storage level, namely
\bes
\varphi_1(\gamma(t)) \leq v(t) \leq \varphi_2(\gamma(t)),
\ees
with $\varphi_1$ and $\varphi_2$
decreasing 
and of class $C^1: [0,1] \rar \RR$, with negative
(resp. positive) values except for
$\varphi_1(0) = \varphi_2(1)=0$,
and for some $\alpha_1>0$ and $\alpha_2>0$:
\be
- \alpha_1 x \leq \varphi_1(x); \quad \varphi_2(x) \leq \alpha_2 ( 1-x).
\ee
In particular we have the uniform bound
\be
v_m :=  \varphi_1(1) \leq
v(t) \leq \varphi_2(0) =: v_M,
\ee
For example, we could take 
$\varphi_1(x) = - \alpha_1 x$ and $\varphi_2(x) = \alpha_2 ( 1-x)$.
Since these constraints imply that the state
remains between 0 and 1 (assuming of course that
$\gamma(0) \in [0,1]$), we can discard the
pure state constraint $\gamma(t) \in [0,1]$.

In what follows we will assume that 
the support of $m_0$ is a compact subset of $(0,1)$. 
It follows that for some $\eps_X>0$, 
any trajectory $(\gamma,v)$ satisfying the mixed state and control constraints 
is such that $\gamma(t) \in [\eps_X,1-\eps_X]$, for all $t \in [0,T]$.
So point 3 of Remark \ref{rem:simplification_assumptions} applies with 
$X= [\eps_X,1-\eps_X]$, taking $\delta \in (0,\eps_X)$ in the definition of
$X'$.

The two mixed constraints are expressed in the format of this paper as
\be
c_1(x,v)= \varphi_1(x)-v; \quad 
c_2(x,v) = v - \varphi_2(x).
\ee
They cannot be active simultaneously, since
$\varphi_1$ and $\varphi_2$ have opposite sign,
and do not have zero value simultaneously. 
It follows that
\bes
\delta := \min_{x\in [0,1]} [
\varphi_2(x) - \varphi_1(x) ) ]
\ees
is positive.
Consequently any $(x,v)$ such that
$c(x,v) \leq 0$ satisfies
\be
\label{c1plusc2}
c_1(x,v) + c_2(x,v) = \varphi_1(x) - \varphi_2(x) \leq - \delta.
\ee
A classical constraint is to have a minimal storage at the
end of the period, say $\gamma(T)\geq \half$,
corresponding to
\be
g_2(x) := \half -x.
\ee
There is no equality constraint on the final state.
Also, $f(x,m)=0$ and $g_0(x,m)= - \pi x$,
where $\pi$ can be interpreted as a final price. 
For $\psi$ we can take the identity, which is not bounded but, since
is applied to a set of bounded controls $v$, we can redefine it as
a bounded, continuous monotone operator.
We can also take for instance $L(x,v) = v^2/2$.

\subsection{Checking hypotheses}

The non obvious hypotheses are 
\hypQualif{}-\hypQualifIneq{}
and \hypQualif{}-\hypQualifInward{}. In view of Remark \ref{rem:simplification_assumptions}, it is enough to check this latter assumption for $x\in X'$.
Assume that $c_1(x,v) \leq c_2(x,v)$, i.e. the second constraint is ``more active''
than the first one.  Taking $w=-\delta/4$, by \eqref{c1plusc2} we have $c_1(x,v) \leq - \delta/2$, 
so that
\be
c_1(x,v) + D_v c_1(x,v) w
\leq -\delta/4.
\ee
Also,
\be
c_2(x,v) + D_v c_2(x,v) w
=
c_2(x,v) - \delta/4 \leq - \delta/4.
\ee  
If $c_2(x,v) \leq c_1(x,v)$, taking $w=\delta/4$ we obtain similar estimates. Hypothesis \hypQualif{}-\hypQualifInward{}
follows.


We next discuss \hypQualif{}-\hypQualifIneq{}.
Remember that we ignore the first condition since
there is no equality constraint for the final state.
Given $\eps=(\eps_1,\eps_2)$ with $\eps_1\geq 0$ and $\eps_2\geq 0$, set
\be
\kappa_1(t):= \varphi_1(\gamma(t)) - v(t) + \eps_1;
\quad
\kappa_2(t) := - \eps_1 + \varphi_2(\gamma(t)) - v(t).
\ee
In our setting, the third condition of
\hypQualif{}-\hypQualifIneq{},
when $\eps_1 := 1/C$, can be expressed as
\be
\label{bound-kappa1}
\kappa_1(t)+ D \varphi_1(\gamma(t)) y(t)
\leq w(t) \leq \kappa_2(t) + D \varphi_2(\gamma(t)) y(t),
\ee
with $y(t) = \int_0^t w(s) \dd s$. Observe that
\be
\label{bound-kappa2}
\kappa_1(t) - \eps_1 \leq 0 \leq \kappa_2(t) + \eps_1.
\ee
Let 
\be
F(t,y) := \min( \kappa_2(t) + D \varphi_2(\gamma(t)) y,\eps_2).
\ee
Since $F(t,y)$ is a Lipschitz function of $y$,
by the Cauchy-Lipschitz theorem, the ODE
\be
\dot y(t) = F(t,y(t)), \;\; t\in (0,T); \quad y(0) =0
\ee
has a unique solution $y_\eps(t)$, and we denote
$w_\eps(t):= \dot y_\eps(t) = F(t,y_\eps(t))$.
By the definition of $F$,
$(y_\eps(t),w_\eps(t))$
satisfies the second linearized mixed constraint.

\begin{lemma}
  We have that:
  \\ {\rm (i)}
  If $\eps=0$, then
  $y_0(t)=0$ for all $t \in [0,T]$.
  \\ {\rm (ii)}
  If   $\eps_1=0$ and $\eps_2>0$, then
$y_\eps(t) \geq 0$ for all $t$, and is always equal to 0
iff the second mixed constraint is always  active.
\\ {\rm (iii)}
If   $\eps_1=0$ and $\eps_2>0$, 
Let $t_0:=\inf\{ t \in [0,T]; \;\; y(t) >0\}$.
If the second mixed constraint is not always active,
then $t_0< T$, and $y(t)>0$, for all $t\in (t_0,T]$.
In particular $y(T)>0$. 
\\ {\rm (iv)}
If   $\eps_1=0$ and $\eps_2>0$,
then the first mixed linearized constraint is never active.
\\ {\rm (v)}
Assume that the second mixed constraint is not always active.
Given $\eps_2>0$, small enough, if $\eps_1>1$ is small enough,
then $(y_\eps,w_\eps)$ satisfies
\hypQualif{}-\hypQualifIneq{}.
\end{lemma}
    
\begin{proof}
(i)
Let $z(t)$ denote the zero function on $[0,T]$. 
Then, in view of
\eqref{bound-kappa2}:
\be
F(t,z(t))) = \min(\kappa_2(t),\eps_2)
=
\min( \kappa_2(t),0) = 0.
\ee
This establishes point (i).
\\ (ii)
Since $F$ is a nondecreasing function of $\eps_2$,
{\em and $y(t)$ is one dimensional},
we deduce that $y_\eps(t)$
is a nondecreasing function of $\eps_2$,
and is therefore nonnegative (since $y_0=0$).
We have  $y_\eps=z$ iff
\be
0 = F(t,0 ) = \min(\kappa_2(t), \eps_2),
\quad \text{for all $t\in [0,T]$,}
\ee
that is, iff
the upper bound is always active.
\\ (iii)
Observe that since $\kappa_2(t) \geq 0$,
and 
$D \varphi_2(\gamma(t)) y_\eps(t) \leq  0$:
\be
\dot y(t)  = F(t,y(t))
\geq \min( D \varphi_2(\gamma(t)) y_\eps(t), \eps_2)
=
D \varphi_2(\gamma(t)) y_\eps(t)
\geq - C y_\eps (t).
\ee
Let $y_\eps (t_1)>0$.
It follows that for $t\in [t_1,T]$,
$y_\eps (t) \geq  e^{-C(t-t_1)} y_\eps(t_1)$.
Point (iii) follows.
\\ (iv)
Since $F$ is a Lipschitz
function of $(y,\eps)$, there exists
$C>0$ such that, for all $\eps\in \RR^2_+$: 
\be
\max_{t\in [0,T]} \left( 
| y_\eps(t)|
+ | D \varphi_1(\gamma(t)) y_\eps(t)|
+ | D \varphi_2(\gamma(t)) y_\eps(t)| \right)
\leq C_0 | \eps |.
\ee
Since $y_\eps(t) \geq 0$ and
$D \varphi_1(\gamma(t)) \leq 0$, we have 
\be
\kappa_1(t)+ D \varphi_1(\gamma(t)) y_\eps(t)
\leq \kappa_1(t) \leq \eps_1.
\ee
We distinguish two cases.
\\ (a)
If $w_\eps(t) = \eps_2$
the result holds,
whenever $\eps_1 <\eps_2$.
\\ (b)
If $w_\eps(t) = \kappa_2(t) +D \varphi_2(\gamma(t)) y_\eps(t)$,
since $\kappa_2(t)-\kappa_1(t)  \geq \delta - 2 \eps_1$, 
we get
\be
w(t) \geq \kappa_1(t) + \delta - 2 \eps_1 - C_0 |\eps|.
\ee
So, the result holds provided that
\be
2 \eps_1 + C_0 |\eps| \leq \delta.
\ee
\\ (v)
Fix $\eps_2>0$ small enough.
For $\eps_1>0$ small enough, we have by a
continuity argument that the first linearized mixed
constraint holds
(as well as the second one by the definition of
$F(t,y)$),
and that $y(T)>0$.
The conclusion follows. 
\end{proof}



\bibliographystyle{abbrv}
\bibliography{hjb,mfg}

\end{document}